\documentclass[10pt]{article}
\usepackage{delarray} 
\usepackage{amscd}
\usepackage{amssymb}
\usepackage{amsmath}
\usepackage{amsthm,amsxtra}
\usepackage{latexsym,amsmath,mathrsfs}
\usepackage[bookmarks,colorlinks]{hyperref}
\usepackage[all]{xy}
\usepackage{amsfonts}
\usepackage{color}
\usepackage{fancybox}
\usepackage{colordvi}
\usepackage{multicol}
\usepackage{textcomp}
\usepackage{stmaryrd}
\usepackage[dvips]{graphicx}

\usepackage{enumerate,xspace}
\usepackage{geometry}
\usepackage{tocbibind}

\newcommand{\Fil}{{\mathrm{Fil}}}

\newcommand{\Tr}{{{\rm Tr}}}

\def\scrO{\mathscr{O}}

\def\val{\mathrm{val}}

\def\an{\mathrm{an}}

\def\cyc{\mathrm{cyc}}

\def\har{\mathrm{har}}

\def\rd{\mathrm{d}}

\def\Tr{\mathrm{Tr}}
\def\Trd{\mathrm{Trd}}

\def\Edg{\mathrm{Edge}}

    \newcommand{\BA}{{\mathbb {A}}} 
    \newcommand{\BC}{{\mathbb {C}}} 
     
    \newcommand{\BG}{{\mathbb {G}}} \newcommand{\BH}{{\mathbb {H}}}

     \newcommand{\BN}{{\mathbb {N}}}
     
    \newcommand{\BQ}{{\mathbb {Q}}} \newcommand{\BR}{{\mathbb {R}}}
     \newcommand{\BT}{{\mathbb {T}}}

     \newcommand{\BZ}{{\mathbb {Z}}}
    \newcommand{\CA}{{\mathcal {A}}} 
     
    \newcommand{\CE}{{\mathcal {E}}} \newcommand{\CF}{{\mathcal {F}}}
    \newcommand{\CG}{{\mathcal {G}}} \newcommand{\CH}{{\mathcal {H}}}
     
     \newcommand{\CL}{{\mathcal {L}}}
    \newcommand{\CM}{{\mathcal {M}}} \newcommand{\CN}{{\mathcal {N}}}
    \newcommand{\CO}{{\mathcal {O}}} 
     
     \newcommand{\CT}{{\mathcal {T}}}
     \newcommand{\CV}{{\mathcal {V}}}
     \newcommand{\CX}{{\mathcal {X}}}
     
     \newcommand{\RB}{{\mathrm {B}}}

     \newcommand{\RP}{{\mathrm {P}}}

    \newcommand{\fP}{{\mathfrak{P}}}

    \newcommand{\End}{{\mathrm{End}}} 
    \newcommand{\Gal}{{\mathrm{Gal}}} \newcommand{\GL}{{\mathrm{GL}}}
    \newcommand{\Hom}{{\mathrm{Hom}}}

       \newcommand{\Lie}{{\mathrm{Lie}}}
        \newcommand{\fn}{{\mathfrak{n}}}
    \newcommand{\fo}{{\mathfrak{o}}}

    \newcommand{\fl}{{\mathfrak{l}}} \newcommand{\fq}{{\mathfrak{q}}}

    \newcommand{\PGL}{{\mathrm{PGL}}} 
     
    \newcommand{\Res}{{\mathrm{Res}}}

      \newcommand{\fp}{{\mathfrak {p}}}

    \renewcommand{\mod}{\hskip 6pt \mathrm{mod} \hskip 3pt}

\newcommand{\dR}{{\mathrm{dR}}}

\newcommand{\cris}{\mathrm{cris}}

 \def\witt{\mathrm{W}}

\def\Gr{\mathrm{Gr}}

\def\witt{\mathrm{W}}

\def\pr{\mathrm{pr}}

\def\Nrd{\mathrm{Nrd}}
\def\der{\mathrm{der}}

\def\fP{\mathfrak{P}}

    \newcommand{\Spec}{{\mathrm{Spec}}}  
     \newcommand{\Sym}{{\mathrm{Sym}}}
     \newcommand{\tr}{{\mathrm{tr}}}
      \newcommand{\ur}{{\mathrm{ur}}}

    \newcommand{\st}{{\mathrm{st}}}

    \newcommand{\itPi}{{\it \Pi}}

    \theoremstyle{plain}

    \newtheorem{thm}{Theorem}[section] \newtheorem{cor}[thm]{Corollary}
    \newtheorem{lem}[thm]{Lemma} 
    \newtheorem{prop}[thm]{Proposition}
    \newtheorem {conj}[thm]{Conjecture}

    \theoremstyle{definition}

    \newtheorem{defn}[thm]{Definition}

    \theoremstyle{remark}
    \newtheorem {rem}[thm]{Remark}
    \newtheorem {ass}[thm]{Assumption}

    \numberwithin{equation}{section}

 \newcommand{\binc}[2]{ \bigg (\!\! \begin{array}{c} #1\\
    #2 \end{array}\!\! \bigg )}

\newcommand{\wvec}[4]{{\scriptsize{\big [ \!\!
\begin{array}{cc} #1 \!\!\! & \!\!\! #2 \\ #3 \!\!\! & \!\!\! #4 \end{array} \!\! \big ] }}}

\begin{document}

\title{$L$-invariants of Hilbert modular forms}
\author{Bingyong Xie \\ \small Department of Mathematics, East China Normal University, Shanghai, China \\ \small byxie@math.ecnu.edu.cn}
\date{}
\maketitle

\begin{abstract} In this paper we show that under certain condition the Fontaine--Mazur
$L$-invariant for a Hilbert eigenform coincides with its Teitelbaum
type $L$-invariant, and thus prove a conjecture of Chida, Mok and
Park.
\end{abstract}

\section*{Introduction}

In the remarkable paper \cite{GS} Greenberg and Stevens proved a
formula  for the derivative at $s=1$ of the $p$-adic $L$-function of
an elliptic curve $E$ over $\BQ$ (i.e. a modular form of weight $2$)
when $p$ is a prime of split multiplicative reduction, which is the
exceptional zero conjecture proposed by Mazur, Tate, and Teitelbaum
\cite{MTT}. An important quantity in this formula is the
 $L$-invariant, namely $\CL(E)=\log_p(q_E)/v_p(q_E)$ where
$q_E\in p \BZ_p$ is the Tate period for $E$.

For higher weight modular forms $f$, since \cite{MTT}, a number of
different candidates for the $L$-invariant $\CL(f)$ have been
proposed. These include:

(1) Fontaine-Mazur's $L$-invariant $\CL_{FM}$ using $p$-adic Hodge
theory,

(2) Teitelbaum's $L$-invariant $\CL_T$ built by the theory of
$p$-adic uniformization of Shimura curves,

(3) an invariant $\CL_C$ by Coleman's theory of $p$-adic integration
on modular curves,

(4) an invariant $\CL_{DO}$ due to Darmon and Orton using
``modular-form valued distributions'',

(5) Breuil's $L$-invariant $\CL_B$ by $p$-adic Langlands theory.

\noindent Now, all of these invariants are known to be equal
\cite{BDI, Breuil, CI, IoSp}. The readers are invited to consult
Colmez's paper \cite{Cz2005} for some historical account.

The exceptional zero conjecture for higher weight modular forms has
been proved by Steven using $\CL_C$ \cite{Stevens}, by
Kato--Kurihara--Tsuji using $\CL_{FM}$ (unpublished), by Orton using
$\CL_{DO}$ \cite{Orton}, by Emerton using $\CL_B$ \cite{Em} and by
Bertolini--Darmon--Iovita using $\CL_T$ \cite{BDI}.

In \cite{Mok} Mok addressed special cases of the exceptional zero
conjecture in the setting of Hilbert modular forms. In \cite{CMP}
Chida, Mok and Park introduced the Teitelbaum type $L$-invariant for
Hilbert modular forms, and conjectured that Teitelbaum type
$L$-invariant coincides with the Fontaine-Mazur $L$-invariant. We
state this conjecture precisely below.

Fix a prime number $p$. Let $F$ be a totally real field, $g=[F:\BQ]$
and $\fp$ a prime ideal of $F$ above $p$. Let $f_\infty$ be a
Hilbert eigen newform with even weight $(k_1, \cdots, k_g, w)$ and
level divided exactly by $\fp$ (i.e. not by $\fp^2$). Here ``even
weight'' means that $k_1,\cdots ,k _g, w$ are all even.

On one hand, by Carayol's result \cite{Car2} we can attach to
$f_\infty$ a $p$-adic representation of
$G_{\BQ}=\mathrm{Gal}(\overline{\BQ}/\BQ)$. This Galois
representation (restricted to $G_{\BQ_p}$) is semistable and thus we
can attach to it a Fontaine-Mazur $L$-invariant
$\CL_{FM}(f_\infty)$.

On the other hand, Chida, Mok and Park attached to an automorphic
form  $\mathbf{f}$ on a totally definite quaternion algebra over $F$
(of the same weight $(k_1, \cdots, k_g,w)$) a Teitelbaum type
$L$-invariant
 $\CL_T(\mathbf{f})$ under the following assumption
\begin{equation*} \hskip -100pt (\text{CMP}) \hskip 100pt
\mathbf{f} \text{ is new at } \fp \text{ and } U_\fp \mathbf{f} =
\CN \fp ^{w/2} \mathbf{f}.
\end{equation*}

Both $\CL_{FM}(f_\infty)$ and $\CL_T(\mathbf{f})$ are vector valued.
See Section \ref{ss:FM-L} and Section \ref{ss:Teitelbaum} for their
precise definitions.

\begin{conj} If $f_\infty$ and $\mathbf{f}$ are associated to each
other by the Jacquet-Langlands correspondence, then
$\CL_{FM}(f_\infty)=\CL_T(\mathbf{f})$.
\end{conj}

Our main result is the following

\begin{thm} \label{thm:main} $($=Theorem \ref{thm:main-text}$)$ Assume that $F$ satisfies the following condition:
$$\text{ there is no place other than } \fp  \text{ above }p. $$

\noindent Let $f_\infty$ and $\mathbf{f}$ be as above. Then
$\CL_{FM}(f_\infty)=\CL_T(\mathbf{f})$.
\end{thm}

We sketch the proof of Theorem \ref{thm:main}. Our method is similar
to that in \cite{IoSp}. The Galois representation attached to
$f_\infty$ comes from the \'etale cohomology $H^1_{\mathrm{et}}$ of
some local system on a Shimura curve. The technical part of our
paper is the computation of the filtered $\varphi_q$-isocrystal
attached to this local system. On one hand, Coleman and Iovita
\cite{CI} provided a precise description of the monodromy operator,
which is helpful for computing Fontaine-Mazur's $L$-invariant. On
the other hand, the Teitelbaum type $L$-invariant is closely related
to the de Rham cohomology of the filtered $\varphi_q$-isocrystal by
Coleman integration and Schneider integration. Our precise
description of the filtered $\varphi_q$-isocrystal allows us to
compute Fontaine-Mazur's $L$-invariant and the Teitelbaum type
$L$-invariant. Finally, analyzing the relation among the monodromy
operator, Coleman integration and Schneider integration finishes the
proof.

When $F$ has more than one place (say $r$ places) above $p$, our
method of computing filtered $\varphi_q$-isocrystals is not valid.
To make it work, one might have to consider the Shimura variety
studied by Rapoport and Zink \cite[Chapter 6]{RZ} (which is of
dimension $r$) instead of the Shimura curve. Coleman and Iovita's
result is valid only for curves, and so can not be applied directly.
We plan to address this problem in a future work.

Our paper is organized as follows. Fontaine-Mazur's $L$-invariant is
introduced in Section \ref{sec:FM-L-inv}. Coleman and Iovita's
result is recall in Section \ref{sec:CoIo}. Section
\ref{sec:univ-spec-mod} is devoted to compute the filtered
$\varphi_q$-isocrystal attached to the universal special formal
module. We introduce various Shimura curves, and study their
$p$-adic uniformizations following Rapoport and Zink respectively in
Section \ref{sec:sh-curves} and Section \ref{sec:p-unif}. In Section
\ref{sec:comp-iso} we use the result in Section
\ref{sec:univ-spec-mod} to determine the filtered
$\varphi_q$-isocrystals attached to various local systems on Shimura
curves. In Section \ref{sec:cover-hodge} we recall the theory of de
Rham cohomology of certain local systems, and in Section
\ref{sec:Teit-L-inv} we recall Chida, Mok and Park's construction of
Teitelbaum type $L$-invariant. Finally in Section \ref{sec:compare}
we combine results in Section \ref{sec:CoIo}, Section
\ref{sec:comp-iso} and Section \ref{sec:cover-hodge}  to prove our
main theorem.

\section*{Acknowledgement} This paper is supported by
the National Natural Science Foundation of China (grant 11671137).

\section*{Notations} For two $\BQ$-algebras $A$ and $B$, write $A\otimes B$ for
$A\otimes_{\BQ}B$. For a ring $R$ let $R^\times$ denote the
multiplicative group of invertible elements in $R$. For a linear
algebraic group over $\BQ$ we will identify it with its $\BQ$-valued
points. 

Let $F$ be a totally real number field, $g=[F:\BQ]$. Let $p$ be a
fixed prime. Suppose that $p$ is inertia in $F$, i.e. there exists
exactly one place of $F$ above $p$, denoted by $\fp$. If $q$ is a
power of $p$, we use $v_p(q)$ to denote $\log_pq$.

Let $\BA_f$ denote $\BQ\otimes_\BZ \widehat{\BZ}$ and let $\BA_f^p$
denote $\BQ\otimes_\BZ (\prod_{\ell\neq p}\BZ_\ell)$. Similarly for
any number field $E$ let $\BA_{E,f}$ denote
$E\otimes_\BZ\widehat{\BZ}$, the group of finite ad\`eles of $E$.


Fix an algebraic closure of $F_\fp$, denoted by $\overline{F_\fp}$,
and let $\BC_p$ be the completion of $\overline{F_\fp}$ with respect
to the $p$-adic topology. By this way we have fixed an embedding $
F_\fp \hookrightarrow \BC_p$. The Galois group
$G_{F_\fp}=\Gal(\overline{F_\fp}/F_\fp)$ can be naturally identified
with the group of continuous $F_\fp$-automorphisms of $\BC_p$.

\section{Fontaine-Mazur invariant} \label{sec:FM-L-inv}

\subsection{Monodromy modules and Fontaine-Mazur $L$-invariant}
\label{sec:mono-fontaine-mazur}

Let  $F_{\fp 0}$ be the maximal absolutely  unramified subfield of
$F_\fp$. Let $q$ be the cardinal number of the residue field of
$F_\fp$.

Let $\RB_\cris, \RB_\st$ and $ \RB_\dR$ be Fontaine's period rings
\cite{Fon}. As is well known, there are operators $\varphi$ and $N$
on $\RB_\st$, and a descending $\BZ$-filtration on $\RB_\dR$;
$\RB_\cris$ is a $\varphi$-stable subring of $\RB_\st$ and $N$
vanishes on $\RB_\cris$. Put $\RB_{\st,
F_\fp}:=\RB_\st\otimes_{F_{\fp 0}} F_\fp$; $\RB_{\st, F_\fp}$ can be
considered as a subring of $\RB_\dR$. We extend the operators
$\varphi_{q}=\varphi^{v_p(q)}$ and $N$ $F_\fp$-linearly to
$\RB_{\st, F_\fp}$.

Let $K$ be either a finite unramified extension of $F_\fp$ or the
completion of the maximal unramified extension of $F_\fp$ in
$\BC_p$. Write $G_K$ for the group of continuous automorphisms of
$\BC_p$ fixing elements of $K$. By our assumption on $K$ we have
$$ (\RB_{\cris, F_\fp})^{G_K} = (\RB_{\st, F_\fp})^{G_K} =
(\RB_{\dR})^{G_K}=K. $$

Let $L$ be a finite extension of $\BQ_p$. For an $L$-linear
representation $V$ of $G_{K}$, we put
$$D_{\st, F_\fp}(V):=(V\otimes_{\BQ_p}\RB_{\st,
F_\fp})^{G_{K}}.$$ This is a finite rank $L\otimes_{\BQ_p}K$-module.
If $V$ is semistable, then $D_{\st, F_\fp}(V)$ is a filtered
$(\varphi_q, N)$-module: the $(\varphi_q,N)$-module structure is
induced from the operators $\varphi_q=1_V \otimes \varphi_q$ and
$N=1_V\otimes N$ on $V \otimes_{\BQ_p} \RB_{\st, F_\fp}$; the
filtration comes from that on $V\otimes_{\BQ_p}\RB_\dR$. Note that
$\varphi_q$ and $N$ are $L\otimes_{\BQ_p}K$-linear.

If $L$ splits $F_\fp$, then $L\otimes_{\BQ_p}K$ is isomorphic to
$\bigoplus_\sigma L\otimes_{\sigma,F_\fp}K$, where $\sigma$ runs
through all embeddings of $F_\fp$ into $L$. Here the subscript
$\sigma$ under $\otimes$ indicates that $F_\fp$ is considered as a
subfield of $L$ via $\sigma$. Let $e_\sigma$ be the unity of the
subring $L\otimes_{\sigma, F_\fp}K$.

We shall need the notion of monodromy modules. This notion is
introduced in \cite{Ma}. However we will use the slightly different
definition given in \cite{IoSp}.

Let $T$ be a finite-dimensional commutative semisimple
$\BQ_p$-algebra. A $T$-object $D$ in the category of filtered
$(\varphi_q,N)$-modules, is called a {\it $2$-dimensional monodromy
$T$-module}, if the following hold:

$\bullet$ $D$ is a free $T_{F_\fp}$-module of rank $2$
($T_{F_\fp}=T\otimes_{\BQ_p}F_\fp$),

$\bullet$ the sequence $D\xrightarrow{N} D \xrightarrow{N} D$ is
exact,

$\bullet$ there exists an integer $j_0$ such that $\Fil^{j_0} D$ is
a free $T_{F_\fp}$-submodule of rank $1$ and $\Fil^{j_0} D \cap
\ker(N)=0$.

\begin{lem} $($\cite[Lemma 2.3]{IoSp}$)$ If $D$ is a monodromy $T$-module, then there exists a
decomposition $D=D^{(1)}\oplus D^{(2)}$ where $D^{(1)}$ and
$D^{(2)}$ are $\varphi_q$-stable free rank one
$T_{F_\fp}$-submodules such that $N:D\rightarrow D$ induces an
isomorphism $N|_{D^{(2)}}: D^{(2)}\xrightarrow{\sim} D^{(1)}$.
\end{lem}

Let $D$ be a monodromy $T$-module and let $j_0$ be as above. The
{\it Fontaine-Mazur $L$-invariant} of $D$ is defined to be the
unique element in $T_{F_\fp}$, denoted as $\CL_{FM}(D)$, such that
$x-\CL_{FM}(D)N(x)\in \Fil^{j_0} D$ for every $x\in D^{(2)}$.

What we are interested in is the case when $T$ is an $L$-algebra,
where $L$ is a field splitting $F_\fp$. Note that we have an
decomposition of $ T_{F_\fp} $:
\begin{equation} \label{eq:decom}
T_{F_\fp} \xrightarrow{\sim} \bigoplus_\sigma T_\sigma, \hskip 10pt
t\otimes a\mapsto (\sigma(a)t)_\sigma,
\end{equation} where $\sigma$ runs through all embeddings of $F_\fp$ in $L$. The
index $\sigma$ in $T_\sigma$ indicates that $T$ is considered as an
$F_\fp$-algebra via $\sigma$. Then we have a decomposition of $D$ by
$D\simeq \bigoplus_\sigma D_\sigma$, where $D_\sigma=e_\sigma \cdot
D$. Each $D_\sigma$ is stable under $\varphi_q$ and $N$. Note that,
for every $j$, $\Fil^j D$ is a $T_{F_\fp}$-submodule. Thus the
filtration on $D$ restricts to a filtration on $D_\sigma$ for each
$\sigma$, and satisfies $\Fil^j D = \bigoplus_\sigma \Fil^j
D_\sigma$ for all $j\in \BZ$.

Using the decomposition (\ref{eq:decom}) we may write $\CL_{FM}(D)$
in the form $(\CL_{FM,\sigma}(D))_\sigma$. It is easy to see that
$\CL_{FM,\sigma}(D)$ is the unique element in $T$ such that
$x-\CL_{FM,\sigma}(D)N(x)\in \Fil^{j_0} D_\sigma$ for every $x\in
D^{(2)}_\sigma$. We also call $(\CL_{FM,\sigma}(D))_\sigma$, a
vector with values in $T$, the {\it Fontaine-Mazur $L$-invariant} of
$D$.

\subsection{Fontaine-Mazur $L$-invariant for Hilbert modular forms}
\label{ss:FM-L}

Let $\{\tau_1,\cdots, \tau_g\}$ be the set of real embeddings
$F\hookrightarrow \BR$. Fix a multiweight $\mathrm{k}=(k_1,\cdots,
k_g,w)\in \BN^{g+1}$ satisfying $k_i\geq 2$ and $k_i\equiv w\mod 2$.

Let $\pi=\otimes_v \pi_v$ be a cuspidal automorphic representation
of $\GL(2,\BA_F)$ such that for each infinite place $\tau_i$, the
$\tau_i$-component $\pi_{\tau_i}$ is a holomorphic discrete series
representation $D_{k_i}$. Let $\fn$ be the level of $\pi$.

Carayol \cite{Car2} attached to such an automorphic representation
(under a further condition) an $\ell$-adic Galois representation,
which will be recalled below.

Let $L$ be a sufficiently large number field of finite degree over
$\BQ$ such that the finite part $\pi^\infty=\otimes_{\fp\nmid
\infty}\pi_\fp$ of $\pi$ admits an $L$-structure $\pi_L^\infty$. The
fixed part $(\pi^\infty_L)^{K_1(\fn)}$ is of dimension $1$ and
generated by an eigenform $f_\infty$. In this case we write
$\pi_{f_\infty}$ for $\pi$.

The local Langlands correspondence associates to every irreducible
admissible representation $\pi$ of $\GL(2, F_\fp)$ defined over $L$
a $2$-dimensional $L$-rational Frobenius-semisimple representation
$\sigma(\pi)$ of the Weil-Deligne group
$WD(\overline{F}_\fp/F_\fp)$. Let $\check{\sigma}(\pi)$ denote the
dual of $\sigma(\pi)$.

For an $\ell$-adic representation $\rho$ of $\Gal(\overline{F}/F)$,
let $\rho_\fp$ denote its restriction to
$\Gal(\overline{F}_\fp/F_\fp)$, $'\hskip -3pt \rho_\fp$ the
Weil-Deligne representation attached to $\rho_\fp$ and $'\hskip
-3pt\rho_\fp^{\text{F-ss}}$ the Frobenius-semisimplification of
$'\hskip -3pt\rho_\fp$.

\begin{thm} \label{thm:car} \cite{Car2} Let $f_\infty$ be an eigenform of multiweight $k$ satisfying
the following condition:
\begin{quote} If $g=[F:\BQ]$ is even, then there exists a finite
place $\fq$ such that the $\fq$-factor $\pi_{f_\infty, \fq}$ lies in
the discrete series.
\end{quote} Then for any prime number $\ell$ and a finite place $\lambda$
of $L$ above $\ell$, there exists a $\lambda$-adic representation
$\rho=\rho_{f_\infty,\lambda}: \Gal(\overline{F}/F)\rightarrow
\GL_{L_\lambda}(V_{f_\infty,\lambda})$ satisfying the following
property:

For any finite place $\fp\nmid \ell$ there is an isomorphism
$$ '\hskip -3pt \rho_{f_\infty,\lambda, \fp}^{\mathrm{F}\text{-}\mathrm{ss}} \simeq \check{\sigma}(\pi_{f_\infty,\fp}) \otimes_L L_\lambda $$
of representations of the Weil-Deligne group
$WD(\overline{F}_\fp/F_\fp)$.
\end{thm}

Saito \cite{Saito} showed that when $\fp|\ell$,
$\rho_{f_\infty,\lambda, \fp}$ is potentially semistable.

Now we assume that $\ell=p$, $\fp$ is the prime ideal of $F$ above
$p$, and $L$ contains $F$. Let $\fP$ be a prime ideal of $L$ above
$\fp$.

\begin{thm} $($=Theorem \ref{thm:semistable}$)$ \label{thm:mono-mod}
Let $f_\infty$ be as in Theorem \ref{thm:car}, $\ell=p$ and
$\lambda=\fP$. If $f_\infty$ is new at $\fp$ (when $[F:\BQ]$ is odd,
we demand that $f_\infty$ is new at another prime ideal), then
$\rho_{f_\infty, \fP, \fp}$ is a semistable $($non-crystalline$)$
representation of $\Gal(\overline{F}_\fp/F_\fp)$, and the filtered
$(\varphi_q,N)$-module $D_{\st, F_\fp}(\rho_{f_\infty,\fP,\fp})$ is
a monodromy $L_\fP$-module.
\end{thm}

\begin{rem} The conditions in Theorem \ref{thm:car} and Theorem
\ref{thm:mono-mod} are used to ensure that via the Jacquet-Langlands
correspondence $f_\infty$ corresponds to a modular form on the
Shimura curve $M$ associated to a quaternion algebra $B$ that splits
at exactly one real place; in Theorem \ref{thm:mono-mod} the
quaternion algebra $B$ is demanded to be ramified at $\fp$. See
Section \ref{ss:shimura-curves} for the construction of $M$.
\end{rem}

Thus $D_{\st, F_\fp}(\rho_{f_\infty,\fP,\fp})$ is associated with
the Fontaine-Mazur $L$-invariant. We define the {\it Fontaine-Mazur
$L$-invariant} of $f_\infty$, denoted by $\CL_{FM}(f_\infty)$, to be
that of $D_{\st,F_\fp}(\rho_{f_\infty,\fP,\fp})$.

\section{Local systems and the associated filtered
$\varphi_q$-isocrystals} \label{sec:CoIo}

Let $X$ be a $p$-adic formal $\scrO_{F_\fp}$-scheme. Suppose that
$X$ is analytically smooth over $\scrO_{F_\fp}$, i.e. the generic
fiber $X^\an$ of $X$ is smooth.


The filtered convergent $\varphi$-isocrystals attached to local
systems are studied in \cite{Faltings, CI}. 
It is more convenience for us to compute the filtered convergent
$\varphi_q$-isocrystals attached to the local systems that we will
be interested in. From now on, we will ignore ``convergent'' in the
notion.

Filtered $\varphi_q$-isocrystal is a natural analogue of filtered
$\varphi$-isocrystal. To define it one needs the notion of
$F_{\fp}$-enlargement. An {\it $F_{\fp}$-enlargement} of $X$ is a
pair $(T, x_T)$ consisting of a flat formal $\scrO_{F_\fp}$-scheme
$T$ and a morphism of formal $\scrO_{F_\fp}$-scheme $x_T:
T_0\rightarrow X$, where $T_0$ is the reduced closed subscheme of
$T$ defined by the ideal $\pi\scrO_T$.

An {\it isocrystal} $\mathscr{E}$ on $X$ consists of the following
data:

$\bullet$ for every $F_\fp$-enlargement $(T,x_T)$ a coherent
$\scrO_T\otimes_{\scrO_{F_\fp}}F_\fp$-module $\mathscr{E}_T$,

$\bullet$ for every morphism of $F_\fp$-enlargements $g: (T',
x_{T'})\rightarrow (T, x_T)$ an isomorphism of

\hskip 12pt $\scrO_{T'}\otimes_{\scrO_{F_\fp}}F_\fp$-modules
$\theta_g: g^* (\mathscr{E}_T)\rightarrow \mathscr{E}_{T'}$.

\noindent The collection of isomorphisms $\{\theta_g\}$ is required
to satisfy certain cocycle condition. If $T$ is an
$F_\fp$-enlargement of $X$, then $\mathscr{E}_T$ may be interpreted
as a coherent sheaf $E_T^\an$ on the rigid space $T^\an$.

As $X$ is analytically smooth over $\scrO_{F_\fp}$, there is a
natural integrable connection $$ \nabla_X: E_X^\an\rightarrow
E_X^\an\otimes \Omega^1_{X^\an} . $$

Note that an isocrystal on $X$ depends only on $X_0$. Let
$\varphi_q$ denote the absolute $q$-Frobenius of $X_0$. A {\it
$\varphi_q$-isocrystal} on $X$ is an isocrystal $\mathscr{E}$ on $X$
together with an isomorphism of isocrystals $\varphi_q:
\varphi_q^*\mathscr{E}\rightarrow \mathscr{E}$. A {\it filtered
$\varphi_q$-isocrystal} on $X$ is a $\varphi_q$-isocrystal
$\mathscr{E}$ with a descending $\BZ$-filtration on $E_X^\an$.


The following well known result compares the de Rham cohomology of a
filtered $\varphi_q$-isocrystal and the \'etale cohomology of the
$\BQ_p$-local system associated to it.

\begin{prop}\label{prop:semistable} \cite[Theorem 3.2]{Faltings}
Suppose that $X$ is a semistable proper curve over $\scrO_{F_\fp}$.
Let $\mathscr{E}$ be a filtered $\varphi_q$-isocrystal over $X$ and
$\CE$ be a $\BQ_p$-local system over $X_{\overline{F}_\fp}$ that are
attached to each other. Then the Galois representation
$H^i_\mathrm{et}(X_{\overline{F}_\fp}, \CE)$ of $G_{F_\fp}$ is
semistable, and the filtered $(\varphi_q,N)$-module $D_{\st,
F_\fp}(H^i_{\mathrm{et}}(X_{\overline{F}_\fp}, \CE))$ is isomorphic
to $H^i_{\mathrm{dR}}(X^{\mathrm{an}}, \mathscr{E})$.
\end{prop}

Coleman and Iovita \cite{CI} gave a precise description of the
monodromy $N$ on $H^1_{\mathrm{dR}}(X^{\mathrm{an}}, \mathscr{E})$.

Now let $X$ be a connected, smooth and proper curve over $F_\fp$
with a regular semistable model $\CX$ over $\scrO_{F_\fp}$ such that
all irreducible components of its special fiber $\overline{\CX}$ are
smooth. For a subset $U$ of $\overline{\CX}$ let $]U[$ denote the
tube of $U$ in $X^\an$. We associate to $\overline{\CX}$ a graph
$\Gr(\overline{\CX})$. Let $\mathrm{n}:
\overline{\CX}^{\mathrm{n}}\rightarrow \overline{\CX}$ be the
normalization of $\overline{\CX}$. The vertices
$\mathrm{V}(\overline{\CX})$ of $\Gr(\overline{\CX})$ are
irreducible components of $\overline{\CX}$. For every vertex $v$ let
$C_v$ be the irreducible component corresponding to $v$. The edges
$\mathrm{E}(\overline{\CX})$ of $\Gr(\overline{\CX})$ are ordered
pairs $\{x,y\}$ where $x$ and $y$ are two different liftings in
$\overline{\CX}^\mathrm{n}$ of a singular point. Let $\tau$ be the
involution on $\mathrm{E}(\overline{\CX})$ such that
$\tau\{x,y\}=\{y,x\}$. Below, for a module $M$ on which $\tau$ acts,
set $M^\pm=\{m\in M: \tau(m)=\pm m\}$.

Let $\mathscr{E}$ be a filtered $\varphi_q$-isocrystal over $X$. For
any $e=\{x,y\}\in\mathrm{E}(\overline{\CX})$ let $H^i_\dR(]e[,
\mathscr{E})$ denote $H^i_\dR(]\mathrm{n}(x)[,\mathscr{E})$. Then
$\tau$ exchanges $H^i_\dR(]e[,\mathscr{E})$ and
$H^i_\dR(]\bar{e}[,\mathscr{E})$ where $\bar{e}=\{y,x\}$. Note that
$\{C_v\}_{v\in \mathrm{V}(\overline{\CX})}$ is an admissible
covering of $X^\an$. From the Mayer-Vietorus exact sequence with
respect to this admissible covering we obtain the following short
exact sequence $$\label{eq:exact-sq} %
\small \xymatrix{ 0 \ar[r] & (\bigoplus_{e\in
\mathrm{E}(\overline{\CX})}H^0_\dR(]e[,\mathscr{E}))^-/
\bigoplus_{v\in
\mathrm{V}(\overline{\CX})}H^0_\dR(]C_v[,\mathscr{E}) \ar[r]^{\hskip
70pt \iota} &  H^1_\dR(X^\an, \mathscr{E})  \\ \ar[r] &
\ker\Big(\bigoplus_{v\in \mathrm{V}(\overline{\CX})}H^1_\dR(]C_v[,
\mathscr{E}) \rightarrow (\bigoplus_{e\in
\mathrm{E}(\overline{\CX})} H^1_\dR(]e[,\mathscr{E}))^+\Big) \ar[r]
& 0.}$$ 
For any $e\in \mathrm{E}(\overline{\CX}) $ there is a
natural residue map $\Res_e: H^1_\dR(]e[, \mathscr{E})\rightarrow
H^0_\dR(]e[, \mathscr{E})$ \cite[Section 4.1]{CI}. These residue
maps induce a map
$$ \bigoplus_{e\in\mathrm{E}(\overline{\CX})}\Res_e: \hskip 10pt  \Big(\bigoplus_{e\in\mathrm{E}(\overline{\CX})}
H^1_\dR(]e[,\mathscr{E})\Big)^+ \rightarrow
\Big(\bigoplus_{e\in\mathrm{E}(\overline{\CX})}
H^0_\dR(]e[,\mathscr{E})\Big)^-. $$

\begin{prop} \label{prop:monodromy} \cite[Theorem 2.6, Remark 2.7]{CI}
The monodromy operator $N$ on $H^1_\dR(X^\an, \mathscr{E})$
coincides with the composition
$$ \iota \circ (\bigoplus_{e\in \mathrm{E}(\overline{\CX})} \Res_e) \circ \Big(H^1_\dR(X^\an, \mathscr{E})
\rightarrow (\bigoplus_{e\in\mathrm{E}(\overline{\CX})}
H^1_\dR(]e[,\mathscr{E}))^+\Big)
$$ where $H^1_\dR(X^\an, \mathscr{E})\rightarrow \Big(\bigoplus_{e\in\mathrm{E}(\overline{\CX})}
H^1_\dR(]e[,\mathscr{E})\Big)^+$ is the restriction map. \end{prop}

\section{The universal special formal module}
\label{sec:univ-spec-mod}

\subsection{Special formal modules and Drinfeld's moduli theorem}

Let $B_\fp$ be the quaternion algebra over $F_\fp$ with invariant
$1/2$. So $B_\fp$ is isomorphic to $F_\fp^{(2)}[\itPi]$;
$\itPi^2=\pi$ and $\itPi a=\bar{a} \itPi$ for all $a\in
F^{(2)}_\fp$. Here, $\pi$ is a fixed uniformizer of $F_\fp$,
$F_\fp^{(2)}$ is the unramified extension of $F_\fp$ of degree $2$,
and $a\mapsto \bar{a}$ denotes the nontrivial $F_\fp$-automorphism
of $F^{(2)}_\fp$.

Let $\scrO_{B_\fp}$ be the ring of integers in $B_\fp$. Let $F_{\fp
0}$ be the maximal absolutely unramified subfield of $F_\fp$, $k$
the residue field of $F_\fp$, and $F_{\fp 0}^{(2)}$ the unramified
extension of $F_{\fp 0}$ of degree $2$.

Let $\scrO^\ur$ denote the maximal unramified extension of
$\scrO_{F_\fp}$, $\widehat{\scrO^\ur}$ its $\pi$-adic completion.
Fix an algebraic closure $\bar{k}$ of $k$. We identify
$\widehat{\scrO^\ur}/\pi\widehat{\scrO^\ur}$ with $\bar{k}$. Then
$\witt(\bar{k})\otimes_{\scrO_{F_\fp^0}}\scrO_{F_\fp}\cong
\widehat{\scrO^\ur}$. Let $\widehat{F_\fp^\ur}$ be the fractional
field of $\widehat{\scrO^\ur}$.

We use the notion of special formal $\scrO_{B_\fp}$-module in
\cite{Drin}.

First we fix a special formal $\scrO_{B_\fp}$-module over $\bar{k}$,
$\Phi$, as in \cite[(3.54)]{RZ}. Let $\iota$ denote the natural
embedding of $F_{\fp 0}$ into $\witt(\bar{k})[1/p]$. Then all
embeddings of $F_{\fp 0}$ into $\witt(\bar{k})[1/p]$ are
$\varphi^{j}\circ \iota$ ($0\leq j\leq v_p(q)-1$). We have the
decomposition
$$ \scrO_{B_\fp} \otimes_{\BZ_p} \witt(\bar{k}) =
\prod_{j=0}^{v_p(q)-1} \scrO_{B_\fp} \otimes_{\scrO_{F^0_\fp},
\varphi^{j}\circ \iota} \witt(\bar{k}) .
$$ Let $u\in \scrO_{B_\fp} \otimes_{\BZ_p}
\witt(\bar{k}) $ be the element whose $\varphi^j\circ
\iota$-component with respect to this decomposition is
$$ u_{\varphi^j\circ \iota} =
\left\{\begin{array}{ll} \itPi \otimes 1 & \text{ if } j=0 , \\ 1
\otimes 1 & \text{ if } j=1, \dots, v_p(q)-1.
\end{array}\right.$$ Let $\widetilde{\mathrm{F}}$ be the
$1\otimes\varphi$-semilinear operator on
$\scrO_{B_\fp}\otimes_{\BZ_p}\witt(\bar{k})$ defined by
$$\widetilde{\mathrm{F}}x = (1\otimes \varphi)x \cdot u, \hskip 10pt x \in \scrO_{B_\fp}\otimes_{\BZ_p}\witt(\bar{k}).
$$ Let $\widetilde{\mathrm{V}}$ be the $1\otimes \varphi^{-1}$-semilinear operator
on $\scrO_{B_\fp}\otimes_{\BZ_p}\witt(\bar{k})$ such that
$\widetilde{\mathrm{F}}\widetilde{\mathrm{V}}=p$. Then
$$ (\scrO_{B_\fp}\otimes_{\BZ_p}\witt(\bar{k}), \widetilde{\mathrm{V}},
\widetilde{\mathrm{F}})$$ is a Dieudonne module over
$\witt(\bar{k})$ with an action of $\scrO_{B_\fp}$ by the left
multiplication. Let $\Phi$ be the special formal
$\scrO_{B_\fp}$-module over $\bar{k}$ whose contravariant Dieudonne
crystal is $(\scrO_{B_\fp}\otimes_{\BZ_p}\witt(\bar{k}),
\widetilde{\mathrm{V}}, \widetilde{\mathrm{F}})$. \footnote{The
Dieudonne crystal in \cite[(3.54)]{RZ} is the covariant Dieudonne
crystal of $\Phi$. The duality between our contravariant Dieudonne
crystal and the covariant Dieudonne crystal is induced by the trace
map $$<\cdot, \cdot>:\scrO_{B_\fp}\times \scrO_{B_\fp}\rightarrow
\BZ_p, (x,y)\mapsto \tr_{F_\fp/\BQ_p} \Big(\delta_{F_\fp/\BQ_p}^{-1}
\tr_{B_\fp/F_\fp} (x y^t)\Big),
$$ where $\delta_{F_\fp/\BQ_p}$ is the different of $F_\fp$ over $\BQ_p$, $\tr_{B_\fp/F_\fp}$ is the reduced trace
map, and $y\mapsto y^t$ is the involution of $B_\fp$ such that
$\itPi^t=\itPi$ and $a^t=\bar{a}$ if $a\in F_\fp^{(2)}$. Then we
have $< b \cdot  x  , y>=<x, b^t \cdot y>$ for any $b\in
\scrO_{B_\fp}$. }

Let $\iota_{0}$ and $\iota_1$ be the extensions of $\iota$ to
$F_{\fp 0}^{(2)}$. Then $$ \varphi^{j}\iota_0, \ \varphi^j \iota_1
\hskip 10pt  (0\leq j\leq v_p(q)-1)$$ are all embeddings of $F_{\fp
0}^{(2)}$ into $\witt(\bar{k})[1/p]$. We have
$$ \scrO_{B_\fp} \otimes_{\BZ_p} \witt(\bar{k}) =
\prod_{j=0}^{v_p(q)-1} \scrO_{B_\fp}
\otimes_{\scrO_{F^{(2)}_{\fp0}}, \varphi^j\circ \iota_0}
\witt(\bar{k}) \times \prod_{j=0}^{v_p(q)-1} \scrO_{B_\fp}
\otimes_{\scrO_{F^{(2)}_{\fp0}}, \varphi^j\circ \iota_1}
\witt(\bar{k}),
$$ where $\scrO_{B_\fp}$ is considered as an $\scrO_{F_\fp^{(2)}}$-module by the left multiplication. Let $X$ be the element of $\scrO_{B_\fp} \otimes_{\BZ_p}
\witt(\bar{k})$ whose $\varphi^j\circ \iota_0$-component ($0\leq
j\leq v_p(q) -1$) is $1\otimes 1$ and whose $\varphi^j\circ
\iota_1$-component
 ($0\leq j\leq v_p(q) -1$) is
$\itPi\otimes 1$. Similarly, let $Y$ be the element whose
$\varphi^j\circ \iota_0$-component ($0\leq j\leq v_p(q)-1$) is
$\itPi \otimes 1$ and whose $\varphi^j\circ \iota_1$-component
($0\leq j\leq v_p(q)-1$) is $\pi \otimes 1$. Then $\{X,Y\}$ is a
basis of $\scrO_{B_\fp} \otimes_{\BZ_p} \witt(\bar{k}) $ over
$\scrO_{F_\fp^{(2)}}\otimes _{\BZ_p} \witt(\bar{k})$.

Note that $\GL(2,F_\fp)=(\End^0_{\scrO_{B_\fp}}\Phi)^\times$
\cite[Lemma 3.60]{RZ}. We normalize the isomorphism such that the
action on the $\varphi$-module
$$(\scrO_{B_\fp}\otimes_{\BZ_p}\witt(\bar{k}),
\widetilde{\mathrm{F}})[1/p]=(B_\fp\otimes_{\BQ_p}\witt(\bar{k})[1/p],
\widetilde{\mathrm{F}})$$ is given by $\wvec{a}{b}{c}{d}X= (a\otimes
1) X+ (c\otimes 1) Y$ and $\wvec{a}{b}{c}{d}Y= (b\otimes 1) X+
(d\otimes 1) Y$.

Let $\widetilde{\mathrm{D}}_0$ denote the  $\varphi_q$-module
$$(B_\fp\otimes_{\BQ_p}\widehat{F_\fp^\ur}, \widetilde{\mathrm{F}}^{v_p(q)})$$ coming from the
$\varphi$-module $(\scrO_{B_\fp}\otimes_{\BZ_p}\witt(\bar{k}),
\widetilde{\mathrm{F}})[1/p]$.

We describe Drinfeld's moduli problem. Let $\mathrm{Nilp}$ be the
category of $\widehat{\scrO^{\ur}}$-algebras on which $\pi$ is
nilpotent. For any $A\in \mathrm{Nilp}$, let $\psi$ be the
homomorphism $\bar{k}\rightarrow A/\pi A$; let $\mathrm{SFM}(A)$ be
the set of pairs $(G, \rho)$ where $G$ is a special formal
$\scrO_{B_\fp}$-module over $A$ and $\rho: \Phi_{A/\pi A}=\psi_*
\Phi \rightarrow G$ is a quasi-isogeny of height zero.

We state a part of Drinfeld's theorem \cite{Drin} as follows.

Let $\CH$ be the Drinfeld upper half plane over $F_\fp$, i.e. the
rigid analytic $F_\fp$-variety whose $\BC_p$-points are
$\BC_p-F_\fp$.

\begin{thm} The functor $\mathrm{SFM}$ is represented by a formal
scheme $\widehat{\CH} \hat{\otimes} \widehat{\scrO^{\ur}}$ over
$\widehat{\scrO^{\ur}}$ whose generic fiber is
$\CH_{\widehat{F_\fp^\ur}}=\CH \hat{\otimes} \widehat{F_\fp^{\ur}}$.
\end{thm}

Let $\CG$ be the universal special formal $\scrO_{B_\fp}$-module
over $\widehat{\CH} \hat{\otimes} \widehat{\scrO^{\ur}}$. There is
an action of $\GL(2,F_\fp)$ on $\CG$ (see \cite[Chapter II
(9.2)]{BC}): The group $\GL(2,F_\fp)$ acts on the functor
$\mathrm{SFM}$ by $g\cdot (\psi; G, \rho)=(\psi\circ
\mathrm{Frob}^{-n}; G, \rho\circ \psi_*(g^{-1}\circ
\mathrm{Frob}^{n}))$ if $v_\fp(\det g)=n$. Here, $v_\fp$ is the
valuation of $\BC_p$ normalized such that $v_\fp(\pi)=1$.

\subsection{The filtered $\varphi_q$-isocrystal attached to the universal special formal
module} \label{ss:univ-fil-phi-mod}

It is rather difficult to describe $\CG$ precisely. \footnote{See
\cite{Teit89} for some information about $\CG$ and \cite{Xie} for a
higher rank analogue.} However, we can determine the associated
(contravariant) filtered $\varphi_q$-isocrystal $\CM$.

In the following, we write $\scrO_{\CH, \widehat{F_\fp^\ur}}$ for
$\scrO_{\CH \hat{\otimes} \widehat{F_\fp^\ur}}$ and $\Omega_{\CH,
\widehat{F_\fp^\ur}}$ for the differential sheaf $\Omega_{\CH
\hat{\otimes} \widehat{F_\fp^{\ur}}}$.

As is observed in \cite{Faltings} and \cite{RZ}, except for the
filtration, the $\varphi_q$-isocrystal $\CM$ is constant. So it is
naturally isomorphic to the $\varphi_q$-isocrystal
$$\widetilde{\mathrm{D}}_0\otimes_{\widehat{F_\fp^\ur}}
\scrO_{\CH, \widehat{F_\fp^\ur}}$$ with the $q$-Frobenius being
$\mathrm{F}^{v_p(q)}\otimes \varphi_{q,\CH_{\widehat{F_\fp^\ur}}}$
and the connection being
$$1\otimes \mathrm{d}:
\widetilde{\mathrm{D}}_0\otimes_{\widehat{F_\fp^{\ur}}}\scrO_{\CH,
\widehat{F_\fp^\ur}}\rightarrow
\widetilde{\mathrm{D}}_0\otimes_{\widehat{F_\fp^{\ur}}}\Omega_{\CH,
\widehat{F_\fp^\ur}}.$$

Next we determine the filtration on $\widetilde{\mathrm{D}}_0
\otimes_{\widehat{F_\fp^\ur}} \scrO_{\CH, \widehat{ F_\fp^\ur }}$.
For any $F_\fp$-algebras $K$ and $ L$, $L\otimes_{\BQ_p}K$ is
isomorphic to $L\otimes_{F_\fp}K \oplus
(L\otimes_{\BQ_p}K)_{\mathrm{non}}$, where
$(L\otimes_{\BQ_p}K)_{\mathrm{non}}$ is the kernel of the
homomorphism $L\otimes_{\BQ_p}K\rightarrow L\otimes_{F_\fp}K$,
$\ell\otimes a\mapsto \ell\otimes a$. If $L$ is a field extension of
$F_\fp$ containing all embeddings of $F_\fp$, then
$L\otimes_{\BQ_p}K=\bigoplus_{\tau: F_\fp\hookrightarrow
L}L\otimes_{\tau, F_\fp} K$ and $(L\otimes_{\BQ_p}K)_{\mathrm{non}}$
corresponds to the non-canonical embeddings. We apply this to
$L=F_\fp$ and $K=\widehat{F^\ur_\fp}$; consider
$\widetilde{\mathrm{D}}_0=B_\fp\otimes_{\BQ_p}\widehat{F^\ur_\fp}$
as an $F_\fp\otimes_{\BQ_p}\widehat{F^\ur_\fp}$-module. Then
$\widetilde{\mathrm{D}}_0$ splits into two parts: one is the
canonical part which corresponds to the natural embedding
$\mathrm{id}: F_\fp\hookrightarrow F_\fp$, and the other is the
non-canonical part. Correspondingly, $ \widetilde{\mathrm{D}}_0
\otimes_{\widehat{F_\fp^\ur}} \scrO_{\CH, \widehat{ F_\fp^\ur }} $
splits into two parts, the canonical part
$B_\fp\otimes_{F_\fp}\scrO_{\CH, \widehat{ F_\fp^\ur }}$ and the
non-canonical part. Since $F_\fp$ acts on the Lie algebra of any
special formal $\scrO_{B_\fp}$-module through the natural embedding,
the filtration on the non-canonical part is trivial.

The filtration on the canonical part is closely related to the
period morphism \cite{Faltings, RZ}. Let us recall the definition of
the period morphism \cite[Section 2.2]{Xie}. We will use the
notations in \cite{Xie}.

Let $M(\Phi)$ be the Cartier module of $\Phi$, a
$\mathbb{Z}/2\mathbb{Z}$-graded module. The
$\mathbb{Z}/2\mathbb{Z}$-grading depends on a choice of
$F_\fp$-embedding of $F_\fp^{(2)}$ into $\widehat{F_\fp^{\ur}}$. We
choose the one, $\tilde{\iota}_0$, that restricts to $\iota_0$, and
denote the other $F_\fp$-embedding by $\tilde{\iota}_{1}$. We fix a
graded $\mathrm{V}$-basis $\{g^0, g^1\}$ of $M(\Phi)$ such that
$\mathrm{V}g^0=\itPi g^0$ and $\mathrm{V} g^1=\itPi g^1$.   Then
$\{g^0, g^1, \mathrm{V} g^0, \mathrm{V} g^1 \}$ is a basis of
$M(\Phi)[1/p]$ over $\widehat{F^\ur_\fp}$; $F_\fp^{(2)}\subset
B_\fp$ acts on $\widehat{F^\ur_\fp} g^0 \oplus \widehat{F^\ur_\fp}
\mathrm{V} g^1$ by $\tilde{\iota}_{0}$, and acts on
$\widehat{F^\ur_\fp} \mathrm{V} g^0 \oplus \widehat{F^\ur_\fp} g^1$
by $\tilde{\iota}_{1}$. See \cite{Drin} for the definition of
Cartier module and the meaning of graded $\mathrm{V}$-basis. See
\cite[(3.55)]{RZ} for the relation between $M(\Phi)$ and the
covariant Dieudonne module attached to $\Phi$. In loc. cit. Cartier
module is called $\tau$-$W_F(L)$-crystal.

Let $R$ be any $\pi$-adically complete
$\widehat{F^\ur_\fp}$-algebra. Drinfeld constructed for each $(\psi;
G, \rho)\in \mathrm{SFM}(R)$ a quadruple $(\eta, T, u, \rho)$. Let
$M=M(G)$ be the Cartier module of $G$, $N(M)$ the auxiliary module
that is the quotient of $M\oplus M$ by the submodule generated by
elements of the form $(\mathrm{V}x, -\itPi x)$ and $\beta_M$ the
quotient map $M\oplus M\rightarrow N(M)$. For $(x_0, x_1)\in M\oplus
M$, we write $((x_0, x_1))$ for $\beta_M(x_0,x_1)$. Then we have a
map $\varphi_M: N(M)\rightarrow N(M)$. See \cite[Definition 4]{Xie}
for its definition.  Put
$$\eta_M:=N(M)^{\varphi_M}, \hskip 10pt T_M:=M/\mathrm{V}M;$$ both $\eta_M$ and $T_M$ are $\mathbb{Z}/2\mathbb{Z}$-graded. Let
$u_M:\eta_M\rightarrow T_M$ be the $\scrO_{F_\fp}[\itPi]$-linear map
of degree $0$
 that is the composition of the inclusion
$\eta_M \hookrightarrow N(M)$ and the map $$ N(M)\rightarrow
M/\mathrm{V}M, \hskip 15pt ((x_0, x_1))\mapsto x_0 \ \mathrm{mod} \
\mathrm{V}M.
$$ Then $\eta_{M(\Phi)}$ is a free $\scrO_{F_\fp}$-module of rank
$4$ with a basis $$\{((g^0, 0)), \ ((g^1,0)) ,((\mathrm{V}g^0, 0)),
((\mathrm{V}g^1, 0))\},$$ where $((g^0, 0))$, $((\mathrm{V}g^1,0))$
are in degree $0$, and $((g^1, 0))$, $((\mathrm{V}g^0,0))$ are in
degree $1$. The quasi-isogeny $\rho: \psi_*\Phi\rightarrow G_{R/\pi
R}$ induces an isomorphism
$$ \rho: \eta^0_{M(\Phi)} \otimes_{\scrO_{F_\fp}}F_\fp \xrightarrow{\sim}
\eta^0_{M(G)} \otimes_{\scrO_{F_\fp}}F_\fp  .$$ Then the period of
$(G, \rho)$ is defined by
\begin{equation}\label{eq:period}
z(G,\rho) = \frac{u'_M\circ \rho((\mathrm{V}g^1,0))}{u'_M\circ
\rho((g^0,0))},
\end{equation}
where $u'_M$ is the the map $\eta^0_{M(G)}\otimes_{\scrO_{F_\fp}}
F_\fp\rightarrow T_M^0\otimes_R R[1/p]$ induced by $u_M$.

Note that considered as a $\varphi_q$-module, $M(\Phi)[1/p]$ is the
dual of $B_\fp\otimes_{F_{\fp}} \widehat{F^\ur_\fp}$, the canonical
part of $\widetilde{\mathrm{D}}_0$. Let $\{v_0, v_1, v_2, v_3\}$ be
the basis of $B_\fp\otimes_{F_{\fp}} \widehat{F^\ur_\fp}$ over
$\widehat{F^\ur_\fp}$ dual to $\{\pi g^1, g^0, \mathrm{V} g^0 ,
\mathrm{V} g^1   \}$. Then
\begin{eqnarray*} \Fil^0
B_\fp\otimes_{F_\fp}\scrO_{\CH, \widehat{ F_\fp^\ur }} &=&
B_\fp\otimes_{F_\fp}\scrO_{\CH, \widehat{ F_\fp^\ur }}
 \\
\Fil^1 B_\fp\otimes_{F_\fp} \scrO_{\CH, \widehat{ F_\fp^\ur }} &=&
\text{the } \scrO_{\CH, \widehat{ F_\fp^\ur }}\text{-submodule
generated by }
\\ && \hskip 20pt \widehat{F^\ur_\fp}\cdot(v_1 +  zv_3) \oplus
\widehat{F^\ur_\fp}\cdot(z v_0 +  v_2)
\\
\Fil^2 B_\fp\otimes_{F_\fp} \scrO_{\CH, \widehat{ F_\fp^\ur }} &=&
0.
\end{eqnarray*}
Here $z$ is the canonical coordinate on $\CH_{\widehat{F_\fp^\ur}}$.

We decompose $B_\fp\otimes_{F_\fp}\widehat{F^\ur_\fp}$ into two
direct summands:
$$ B_\fp\otimes_{F_\fp}\widehat{F^\ur_\fp} = B_\fp\otimes_{F_\fp^{(2)},
\ \tilde{\iota}_{0}}\widehat{F^\ur_\fp} \oplus
B_\fp\otimes_{F_\fp^{(2)}, \ \tilde{\iota}_{1}} \widehat{F^\ur_\fp},
$$ where $B_\fp$ is considered as an $F_\fp^{(2)}$-module by left multiplication.
Let $e_0$ and $e_1$ denote the projection to the first summand and
that to the second, respectively.  We may choose $g^0, g^1$ such
that $v_0=e_0 X , \ v_1= e_1 Y, \ v_2= e_0 Y , \ v_3= e_1 X$. Thus
\begin{eqnarray*} \Fil^0
B_\fp\otimes_{F_\fp}\scrO_{\CH, \widehat{ F_\fp^\ur }} &=&
B_\fp\otimes_{F_\fp}\scrO_{\CH, \widehat{ F_\fp^\ur }}  \\
\Fil^1 B_\fp\otimes_{F_\fp} \scrO_{\CH, \widehat{ F_\fp^\ur }} &=&
\text{the } F_\fp^{(2)}\otimes_{F_\fp}\scrO_{\CH, \widehat{
F_\fp^\ur }}\text{-submodule generated by } zX + Y
\\
\Fil^2 B_\fp\otimes_{F_\fp} \scrO_{\CH, \widehat{ F_\fp^\ur }} &=&
0.
\end{eqnarray*}

Finally we note that the induced action of $\mathrm{GL}(2, F_\fp)$
on $\CH$ is given by $\wvec{a}{b}{c}{d} z = \frac{az+b}{cz+d}$.

\section{Shimura curves} \label{sec:sh-curves}

Fix a real place $\tau_1$ of $F$.  Let $B$ be a quaternion algebra
over $F$ that splits at $\tau_1$ and is ramified at other real
places $\{\tau_2,\cdots, \tau_g\}$ and  $\fp$.

\subsection{Shimura curves $M$, $M'$ and $M''$}
\label{ss:shimura-curves}

We will use three Shimura curves studied by Carayol \cite{Car} and
recall their constructions below (see also \cite{Saito}).

Let $G$ be the reductive algebraic group over $\BQ$ such that
$G(R)=(B\otimes R)^\times$ for any $\BQ$-algebra $R$. Let $Z$ be the
center of $G$; it is isomorphic to $T=\Res_{F/\BQ}\BG_m$.  Let
$\nu:G\rightarrow T$ be the morphism obtained from the reduced norm
$\mathrm{Nrd}_{B/F}$ of $B$. The kernel of $\nu$ is $G^{\der}$, the
derived group of $G$, and thus we have a short exact sequence of
algebraic groups
$$\xymatrix{ 1 \ar[r] & G^\der \ar[r] & G \ar[r]^{\nu} & T\ar[r] & 1.  }$$

Let $X$ be the $G(\BR)$-conjugacy class of the homomorphism
\begin{eqnarray*} h: \hskip 50pt \BC^\times & \rightarrow & G(\BR)
=\GL_2(\BR)\times \BH^\times  \times
\cdots \times \BH^\times \\
z= x+\sqrt{-1}y & \mapsto & \hskip 45pt
\left(\wvec{x}{y}{-y}{x}^{-1}, \  1, \ \ \cdots, \ \ \ 1\right),
\end{eqnarray*} where $\BH$ is the Hamilton quaternion algebra. The conjugacy class $X$ is naturally identified with
the union of upper and lower half planes. Let
$M=M(G,X)=(M_U(G,X))_U$ be the canonical model of the Shimura
variety attached to the Shimura pair $(G,X)$; the canonical model is
defined over $F$, the reflex field of $(G,X)$. There is a natural
right action of $G(\BA_f)$ on $M(G,X)$. Here and in what follows, by
abuse of terminology we call a projective system of varieties simply
a variety.

Take an imaginary quadratic field $E_0=\BQ(\sqrt{-a})$ ($a$ a
square-free positive integer) such that $p$ splits in $E_0$. Put
$E=FE_0$ and $D=B\otimes_FE=B\otimes_{\BQ}E_0$. We fix a square root
$\rho$ of $-a$ in $\BC$. Then the prolonging of $\tau_i$ to $E$ by
$x+ y \sqrt{-a} \mapsto \tau_i(x) + \tau_i(y) \rho$ (resp. $x+ y
\sqrt{-a} \mapsto \tau_i(x) - \tau_i(y) \rho$) is denoted by
$\tau_i$ (resp. $\bar{\tau}_i$).

Let $T_E$ be the torus $\Res_{E/\BQ}\BG_m$, $T_E^1$ the subtorus of
$T_E$ such that $T_E^1(\BQ)=\{z\in E: z\bar{z}=1\}$. We consider the
amalgamate product $G''=G\times_Z T_E$, and the morphism
$G''=G\times_Z T_E\xrightarrow{\nu''}T''=T\times T_E^1$ defined by
$(g,z)\mapsto (\nu(g)z\bar{z},z/\bar{z})$. Consider the subtorus
$T'=\BG_m\times T_E^1$ of $T''$, and let $G'$ be the inverse image
of $T'$ by the map $\nu''$. The restriction of $\nu''$ to $G'$ is
denoted by $\nu'$. Both the derived group of $G'$ and that of $G''$
are identified with $G^\der$, and we have two short exact sequences
of algebraic groups
$$ \xymatrix{ 1 \ar[r] & G^\der \ar[r] & G' \ar[r]^{\nu'} & T'\ar[r] & 1}
$$ and $$\xymatrix{
1 \ar[r] & G^\der \ar[r] & G'' \ar[r]^{\nu''} & T''\ar[r] & 1. }$$

The complex embeddings $\tau_1,\cdots, \tau_g$ of $E$ identify
$G''(\BR)$ with $\GL_2(\BR)\cdot \BC^\times \times \BH^\times \cdot
\BC^\times \times \cdots \times \BH^\times \cdot \BC^\times$. We
consider the $G'(\BR)$-conjugacy class $X'$ (resp.
$G''(\BR)$-conjugacy class $X''$) of the homomorphism
\begin{eqnarray*} h': \hskip 40pt \BC^\times & \rightarrow & G'(\BR)
\subset \\
&& G''(\BR)=\GL_2(\BR)\cdot \BC^\times \times \BH^\times \cdot
\BC^\times \times
\cdots \times \BH^\times \cdot \BC^\times \\
z= x+\sqrt{-1}y & \mapsto & \hskip 50pt
\left(\wvec{x}{y}{-y}{x}^{-1}\otimes 1, \ \ 1\otimes z^{-1}, \ \
\cdots, \ \ 1\otimes z^{-1}\right).
\end{eqnarray*}
Let $M'=M(G',X')$ and $M''=M(G'',X'')$ be the canonical models of
the Shimura varieties defined over their reflex field $E$.  There
are natural right actions of $G'(\BA_f)$ and $G''(\BA_f)$ on $M'$
and $M''$, respectively.

Put $T_{E_0}=\Res^{E_0}_{\BQ}\BG_m$. Using the complex embeddings
$\tau_1,\cdots, \tau_g$ of $E$, we identify $T_E(\BR)$ with
$\BC^\times\times \cdots \times \BC^\times$; similarly via the
embedding $x+y\sqrt{-a}\rightarrow x+y\rho$ we identify
$T_{E_0}(\BR)$ with $\BC^\times$. Consider the homomorphisms
\begin{eqnarray*} h_E &:& \BC^\times \rightarrow T_E(\BR)=\BC^\times
\times \cdots \times \BC^\times,  \ \ z\mapsto (z^{-1},
1,\cdots, 1), \\
h_{E_0} &:& \BC^\times \rightarrow T_{E_0}(\BR)=\BC^\times, \ \
\hskip 55pt z\mapsto z^{-1}.\end{eqnarray*} Let $N_E=M(T_E, h_E)$
and $N_{E_0}=M(T_{E_0}, h_{E_0})$ be the canonical models attached
to the pairs $(T_E,h_E)$ and $(T_{E_0}, h_{E_0})$ respectively. Then
$N_E$ is defined over $E$, and $N_{E_0}$ is defined over $E_0$.

Consider the homomorphism $\alpha: G\times T_E\rightarrow G''$ of
algebraic groups inducing
$$B^\times \times E^\times \rightarrow G''(\BQ)\subset (B\otimes_\BQ E)^\times, \ \ (b, e)\mapsto b
\otimes \mathrm{N}_{E/E_0}(e)e^{-1}$$ on $\BQ$-valued points, and
the homomorphism $\beta:G\times T_E\rightarrow T_{E_0}$ inducing
$$\mathrm{N}_{E/E_0}\circ \pr_2: B^\times \times E^\times\rightarrow
E_0^\times$$ on $\BQ$-valued points. Here, $\mathrm{N}_{E/E_0}$
denotes the norm map $E^\times\rightarrow E_0^\times$. Since
$h'=\alpha\circ (h\times h_E)$ and $h_{E_0}=\mathrm{N}_{E/E_0}\circ
h_E$, $\alpha$ and $\beta$ induce morphisms of Shimura varieties
$M\times N\rightarrow M''$ and $M\times N_E\rightarrow N_{E_0}$
again denoted by $\alpha$ and $\beta$ respectively. We have the
following diagram
$$\xymatrix{ M & M\times N_E \ar[d]^{\beta} \ar[l]_{\pr_1}\ar[r]^{\alpha} & M'' & M'\ar[l] \\
& N_{E_0} . &&  }$$

\subsection{Connected components of $M$, $M\times N_E$, $M'$ and $M''$}

We write $\widetilde{G}$ for $G\times T_E$ and write $\widetilde{M}$
for $M\times N_E$.  For $\natural= \widetilde{} , \emptyset, \ ',
''$, as $B$ is ramified at $\fp$,  there exists a unique maximal
compact open subgroup $U^\natural_{p,0}$ of $G^\natural(\BQ_p)$.
Then $U'_{p,0}=U''_{p,0}\cap G'(\BQ_p)$ and
$U''_{p,0}=\alpha(\widetilde{U}_{p,0})$.

If $U^\natural$ is a subgroup of $G^\natural(\BA_f)$ of the form
$U^\natural_{p,0}U^{\natural,p}$ where $U^{\natural,p}$ is a compact
open subgroup of $G^\natural(\BA_f^p)$, we will write
$M^\natural_{0,U^{\natural,p}}$ for $M^\natural_{U^\natural}$. Let
$M^\natural_0$ denote the projective system $(M^\natural_{0,
U^{\natural,p}})_{U^{\natural,p}}$; this projective system has a
natural right action of $G^\natural(\BA_f^p)$.

\begin{lem} \label{lem:geom-conn}
\begin{enumerate}
\item \label{it:geom-conn} For any sufficiently small
$U^{\natural,p}$, each geometrically connected component of
$M^\natural_{0, U^{\natural,p}}$ is defined over a field that is
unramified at all places above $p$.
\item \label{it:iso-on-geom-conn} Let
$\widetilde{U}^{p}$ be a sufficiently small compact open subgroup of
$\widetilde{G}(\BA_f^p)$. Then the morphism
$$ \widetilde{M}_{0, \widetilde{U}^{p}} \rightarrow M''_{0, \alpha(\widetilde{U}^{p})}
$$
induced by $\alpha$ is an isomorphism onto its image when restricted
to every geometrically connected component.
\end{enumerate}
\end{lem}
\begin{proof}
When $U^{\natural,p}$ is sufficiently small, $M^\natural_{0,
U^{\natural,p}}$ is smooth.  Let $\pi_0(M^\natural_{0,
U^{\natural,p}})$ denote the group of geometrically irreducible
components of $ M^\natural_{0, U^{\natural,p}}$ over
$\overline{\BQ}$ (which must be connected since $M^\natural_{0,
U^{\natural,p}}$ is smooth). Then $\Gal(\overline{\BQ}/E)$ acts on
$\pi_0(M^\natural_{0, U^{\natural,p}})$. This action is explicitly
described by Deligne \cite[Theorem 2.6.3]{Del}, from which we deduce
(\ref{it:geom-conn}).

As $\alpha$ induces an isomorphism from the derived group of
$\widetilde{G}$ to that of $G''$, by \cite[Remark 2.1.16]{Del} or
\cite[Proposition II.2.7]{Mil} we obtain
(\ref{it:iso-on-geom-conn}).
\end{proof}

\subsection{Modular interpolation of $M'$}

Let $\ell\mapsto \bar{\ell}$ be the involution on
$D=B\otimes_{\BQ}E_0$ that is the product of the canonical
involution on $B$ and the complex conjugate on $E_0$. Choose an
invertible symmetric element $\delta\in D$ ($\delta=\bar{\delta}$).
Then we have another involution $\ell\mapsto
\ell^*:=\delta^{-1}\bar{\ell}\delta$ on $D$.

Let $V$ denote $D$ considered as a left $D$-module. Let $\psi$ be
the non-degenerate alternating form on $V$ defined by $ \psi(x,y)
=\Tr_{E/\BQ} (\sqrt{-a} \ \Trd_{D/E}(x \delta y^*)) $, where
$\Tr_{E/\BQ}$ is the trace map and $\Trd_{D/E}$ is the reduced trace
map. For $\ell\in D$ put
$$t(\ell)=\tr(\ell; V_\BC/\Fil^0 V_\BC)$$ where $\Fil^\bullet$ is
the Hodge structure defined by $h'$. We have
$$ t(\ell)=(\tau_1+\bar{\tau}_1 +2\tau_2+\cdots+2\tau_g)(\tr_{D/E}(\ell)) $$
for $\ell\in D$.  The subfield of $\BC$ generated by $t(\ell)$,
$\ell\in D$, is exactly $E$.

Choose an order $\scrO_D$ of $D$, $T$ the corresponding lattice in
$V$. With a suitable choice of $\delta$, we may assume that
$\scrO_D$ is stable by the involution $\ell\mapsto \ell^*$ and that
$\psi$ takes integer values on $T$. Put
$\hat{\scrO}_D:=\scrO_D\otimes \hat{\BZ}$ and $\hat{T}:=T\otimes
\hat{\BZ}$.

In Section \ref{sec:p-unif} when we consider the $p$-adic
uniformization of the Shimura curves, we need to make the following
assumption.

\begin{ass} \label{ass} We assume that  $\delta$ is chosen such that $\hat{T}$ is
stable by $U'_{p,0}$.
\end{ass}


If $U'$ is a sufficiently small compact open subgroup of $G'(\BA_f)$
keeping $\hat{T}$, then $M'_{U'}$ represents the following functor
$\CM_{U'}$ \cite[Section 5]{Saito}:

For any $E$-algebra $R$, $\CM_{U'}(R)$ is the set of isomorphism
classes of quadruples $(A, \iota, \theta, \kappa)$ where

$\bullet$ $A$ is an isomorphism class of abelian schemes over $R$
with an endomorphism $\iota:\scrO_D\rightarrow \End(A)$ such that
$\tr(\iota(\ell), \Lie A)=t(\ell)$ for all $\ell\in \scrO_D$.

$\bullet$ $\theta$ is a polarization $A\rightarrow \check{A}$ whose
associated Rosati involution sends $\iota(\ell)$ to $\iota(\ell^*)$.

$\bullet$ $\kappa$ is a $U'$-orbit of $\scrO_D\otimes
\hat{\BZ}$-linear isomorphisms $\hat{T}(A):=\prod\limits_\ell
T_\ell(A) \rightarrow \hat{T}$ such that there exists a
$\hat{\BZ}$-linear isomorphism $\kappa': \hat{T}(1)\rightarrow
\hat{\BZ}$ making the diagram
$$ \xymatrix{ \hat{T}(A) \times \hat{T}(A) \ar[r]^{(1,\theta_*)} \ar[d]^{\kappa\times \kappa } &
\hat{T}(A)\times \hat{T}(\check{A}) \ar[r] & \hat{\BZ}(1)\ar[d]^{\kappa'} \\
\hat{T}\times \hat{T} \ar[rr]^{\psi\otimes\: \hat{\BZ}} && \hat{\BZ}
 } $$
commutative.

Let $\CA_{U'}$ be the universal $\scrO_D$-abelian scheme over
$M'_{U'}$.

\section{$p$-adic Uniformizations of Shimura curves}
\label{sec:p-unif}

\subsection{Preliminaries}

We provide two simple facts, which will be useful later.

(i) Let $X$ be a scheme with a discrete action of a group $C$ on the
right hand side, and let $Z$ be a group that contains $C$ as a
normal subgroup of finite index. Fix a set of representatives
$\{g_i\}_{i\in C \backslash Z}$ of $C\backslash Z$ in $Z$. We define
a scheme $X*_CZ$ with a right action of $Z$ below. As a scheme
$X*_CZ$ is $\bigsqcup_{C\backslash Z} X^{(g_i)}$, where $X^{(g_i)}$
is a copy of $X$. For any $g\in Z$ and $x^{(g_i)}\in X^{(g_i)}$, if
$g_ig=h g_k$ with $h\in C$, then $x^{(g_i)}\cdot g = (x\cdot
h)^{(g_k)}$. It is easy to show that up to isomorphism $X*_CZ$ and
the right action of $Z$ are independent of the choice of $\{ g_i
\}_{i\in C\backslash Z}$.

(ii) Let $X_1$ and $X_2$ be two schemes whose connected components
are all geometrically connected. Suppose that each of $X_1$ and
$X_2$ has an action of an abelian group $Z$; $Z$ acts freely on the
set of components of $X_1$ (resp. $X_2$).  Let $C$ be a closed
subgroup of $Z$. Then the $Z$-actions on $X_1$ and $X_2$ induce
$Z/C$-actions on $X_1/C$ and $X_2/C$.

\begin{lem} \label{lem:prelim}
If there exists a $Z/C$-equivariant isomorphism $\gamma: X_1/C
\rightarrow X_2/C$, then there exists a $Z$-equivariant isomorphism
$\tilde{\gamma}: X_1\rightarrow X_2$ such that the following diagram
\[ \xymatrix{ X_1 \ar[r]^{\tilde{\gamma}} \ar[d]^{\pi_1} & X_2 \ar[d]^{\pi_2} \\  X_1/C \ar[r] ^{\gamma} & X_2/C } \]
is commutative, where $\pi_1$ and $\pi_2$ are the natural
projections.
\end{lem}
\begin{proof}
We identify $X_1/C$ with $X_2/C$ by $\gamma$, and write $Y$ for it.
The condition on $Z$-actions implies that the action of $Z/C$ on the
set of connected components of $Y$ is free and that the morphism
$\pi_1$ (resp. $\pi_2$) maps each connected component of $X_1$
(resp. $X_2$) isomorphically to its image.

We choose a set of representatives $\{Y_i\}_{i\in I}$ of the
$Z/C$-orbits of components of $Y$. Then $\{\bar{g}Y_i:  \bar{g} \in
 Z/C, i\in I\} $ are all different connected components of $Y$. For
each $i\in I$ we choose a connected component $\tilde{Y}_i^{(1)}$
(resp. $\tilde{Y}_i^{(2)}$) of $X_1$ (resp. $X_2$) that is a lifting
of $Y_i$. Then $\{g \tilde{Y}_i^{(1)}: g \in Z, i\in I \}$ (resp. $
\{ g \tilde{Y}_i^{(2)}: g \in Z, i\in I \} $) are all different
connected components of $X_1$ (resp. $X_2$).

As $ \pi_1|_{\tilde{Y_i}^{(1)}}: \tilde{Y_i}^{(1)}\rightarrow Y_i $
and $ \pi_2|_{\tilde{Y_i}^{(2)}}: \tilde{Y_i}^{(2)}\rightarrow Y_i $
are isomorphisms, there exists an isomorphism $ \tilde{\gamma}_i:
\tilde{Y_i}^{(1)}\rightarrow \tilde{Y_i}^{(2)} $ such that $
\pi_1|_{\tilde{Y_i}^{(1)}} = \pi_2|_{\tilde{Y_i}^{(2)}} \circ
\tilde{\gamma}_i $. We define the morphism $\tilde{\gamma}:
X^{(1)}\rightarrow X^{(2)}$ as follows: $ \tilde{\gamma} $ maps $ g
\tilde{Y}_i^{(1)} $ to $ g \tilde{Y}_i^{(2)} $, and $
\tilde{\gamma}|_{g \tilde{Y}_i^{(1)} }= g\circ \tilde{\gamma}_i
\circ g^{-1} $. Then $\tilde{\gamma}$ is a $Z$-equivariant
isomorphism and $ \pi_1 = \pi_2\circ\tilde{\gamma}$.
\end{proof}

\subsection{Some Notations} \label{ss:notations}
Fix an isomorphism $\BC\cong \BC_p$. 
Combining the isomorphism $\BC\cong \BC_p$ and
the inclusion $E_0\hookrightarrow \BC$, $x+y\sqrt{-a}\mapsto
x+y\rho$, we obtain an inclusion $E_0\hookrightarrow \BQ_p$ and
$E\hookrightarrow F_\fp$. Thus $D\otimes \BQ_p$ is isomorphic to
$B_\fp\oplus B_\fp$.

Note that $G(\BQ_p)$ is isomorphic to $B_\fp^\times$, $G'(\BQ_p)$ is
isomorphic to the subgroup $$\{ (a,b): a,b \in B_\fp^\times,
\bar{a}b\in \BQ_p^\times \}$$ of $B_\fp^\times\times B_\fp^\times$,
and $G''(\BQ_p)$ is isomorphic to $$\{ (a,b): a,b \in B_\fp^\times,
\bar{a}b\in F_\fp^\times \},$$ where $a\mapsto \bar{a}$ is the
canonical involution on $B$. Note that $T_{E}(\BQ_p)$ is isomorphic
to $F_\fp^\times\times F_\fp^\times$, and $T_{E_0}(\BQ_p)$ is
isomorphic to $\BQ_p^\times \times \BQ_p^\times$. We normalize these
isomorphisms such that $G'(\BQ_p)\hookrightarrow G''(\BQ_p)$ becomes
the natural inclusion
$$ \{ (a,b): a,b \in B_\fp^\times,
\bar{a}b\in \BQ_p^\times \} \hookrightarrow \{ (a,b): a,b \in
B_\fp^\times, \bar{a}b\in F_\fp^\times \} ,$$ $\alpha:
G(\BQ_p)\times T_E(\BQ_p)\rightarrow G''(\BQ_p)$ becomes
\begin{eqnarray*} B_\fp^\times \times (F_\fp^\times \times
F_\fp^\times) &\rightarrow & \{ (a,b): a,b \in B_\fp^\times,
\bar{a}b\in F_\fp^\times \} \\ (a, (t_1, t_2)) &\mapsto&
(a\frac{\mathrm{N}_{F_\fp/\BQ_p}(t_1)}{t_1}, a
\frac{\mathrm{N}_{F_\fp/\BQ_p}(t_2)}{t_2}) ,
\end{eqnarray*} and $\beta: G(\BQ_p)\times T_E(\BQ_p)\rightarrow
T_{E_0}(\BQ_p)$ becomes \begin{eqnarray*} B_\fp^\times \times
(F_\fp^\times \times F_\fp^\times) & \rightarrow & \BQ_p^\times
\times \BQ_p^\times \\ (a, (t_1, t_2)) &\mapsto &
(\mathrm{N}_{F_\fp/\BQ_p}(t_1), \mathrm{N}_{F_\fp/\BQ_p}(t_2)) .
\end{eqnarray*}

Let $\bar{B}$ be the quaternion algebra over $F$ such that
$$ \mathrm{inv}_v(\bar{B}) = \left\{ \begin{array}{ll} \mathrm{inv}_v(B)
& \text{ if }v \neq \tau_1, \fp , \\ \frac{1}{2} & \text{ if } v=\tau_1, \\
0 & \text{ if }v=\fp . \end{array}\right. $$ With $\bar{B}$ instead
of $B$ we can define analogues of $G$, $G'$ and $G''$, denoted by
$\bar{G}$, $\bar{G}'$ and $\bar{G}''$ respectively. For $\natural=
\emptyset, ',''$ we have $ \bar{G}^\natural (\BA_f^p) = G^\natural
(\BA_f^p) $; $\bar{G}(\BQ_p)$ is isomorphic to $\GL(2,F_\fp)$;
$\bar{G}'(\BQ_p)$ is isomorphic to the subgroup
$$\{ (\wvec{a_1}{b_1}{c_1}{d_1},\wvec{a_2}{b_2}{c_2}{d_2}):  a_i,b_i, c_i, d_i \in  F_\fp ,
\ \wvec{d_1}{-b_1}{-c_1}{a_1} \wvec{a_2}{b_2}{c_2}{d_2}\in
\BQ_p^\times \}$$ of $\GL(2,F_\fp)\times \GL(2,F_\fp)$, and
$\bar{G}''(\BQ_p)$ is isomorphic to  $$\{
\wvec{a_1}{b_1}{c_1}{d_1},\wvec{a_2}{b_2}{c_2}{d_2}):  a_i,b_i, c_i,
d_i \in  F_\fp , \ \wvec{d_1}{-b_1}{-c_1}{a_1}
\wvec{a_2}{b_2}{c_2}{d_2} \in F_\fp^\times \}.$$

If $\natural=\emptyset$, let $\bar{G}(\BQ_p)$ act on
$\CH_{\widehat{F^\ur_\fp}}$ as in Section \ref{sec:univ-spec-mod}.
If $\natural=\widetilde{}$, let $\widetilde{\bar{G}}=\bar{G}\times
T_E$ act on $\CH_{\widehat{F^\ur_\fp}}$ by the projection to the
first factor. If $\natural='$ or $''$, let $\bar{G}^\natural(\BQ_p)$
act on $\CH_{\widehat{F^\ur_\fp}}$ by the first factor. Let
$\bar{G}^\natural(\BQ)$ act on $\CH_{\widehat{F^\ur_\fp}}$ via its
embedding into $\bar{G}^\natural(\BQ_p)$.

The center of $\bar{G}^\natural$, $Z(\bar{G}^\natural)$, is
naturally isomorphic to the center of $G^\natural$, $Z(G^\natural)$;
we denote both of them by $Z^\natural$.

\subsection{The $p$-adic uniformizations} \label{ss:p-adic-unif}

Let $\natural$ be either $\widetilde{}, \ '$ or $''$. For any
compact open subgroup $U^{\natural, p}$ of $G^\natural(\BA_f^p)$,
let $X^\natural_{U^{\natural,p}}$ denote $M^\natural_{0,
U^{\natural,p}}\times_{\Spec(F_\fp)} \Spec(\widehat{F_\fp^\ur})$.

\begin{prop}\label{prop:p-adic-unif} Suppose that Assumption
\ref{ass} holds.
\begin{enumerate}
\item\label{it:p-adic-unif-a} Assume that $\natural=\widetilde{} \ , '$ or $''$.
For any sufficiently small compact open subgroup $U^{\natural,p}$ of
$G^\natural(\BA_f^p)$, writing $U^\natural=U^\natural_{p,0}
U^{\natural p}$, we have a $Z^\natural(\BQ)\backslash
Z^\natural(\BA_f)/ (Z^\natural(\BA_f) \cap U^\natural) $-equivariant
isomorphism \begin{equation} \label{eq:linshi}
X^\natural_{U^{\natural, p}} \cong \bar{G}^\natural(\BQ) \backslash
(\CH_{\widehat{F^\ur_\fp}}\times G^\natural(\BA_f)/U^\natural).
\end{equation}
Here, $\bar{G}^\natural(\BQ)$ acts on $\CH_{\widehat{F^\ur_\fp}}$ as
mentioned above and acts on $G^\natural(\BA_f^p)/U^{\natural, p}$ by
the embedding $\bar{G}^\natural(\BQ)\hookrightarrow
\bar{G}^\natural(\BA_f^p)\xrightarrow{\simeq} G^\natural(\BA_f^p)$;
in the case of $\natural='$ or $''$, if $g\in \bar{G}^\natural(\BQ)$
satisfies $g_p=(a, b)$ with $a, b\in \GL(2,F_\fp)$, then $g$ acts on
$G^\natural(\BQ_p)/U^\natural_{p,0}$ via the left multiplication by
$(\itPi^{v_{\fp}(\det a)}, \itPi^{v_\fp(\det b)})$; while, in the
case of $\natural=\widetilde{}$, $\widetilde{g} = (g , t) \in
\widetilde{\bar{G}}(\BQ)$ $(g\in G(\BQ), t\in T_E(\BQ))$ acts on
$\widetilde{G}(\BQ_p)/\widetilde{U}_{p,0}$ via the left
multiplication by $(\itPi^{v_\fp(\det g_p)}, t_p)$; the group
$Z^\natural(\BQ)\backslash Z^\natural(\BA_f)/ (Z^\natural(\BA_f)
\cap U^\natural)$  acts on the right hand side of
$($\ref{eq:linshi}$)$ by right multiplications on
$\bar{G}^\natural(\BA_f)$.
\item \label{it:p-adic-unif-b} The isomorphisms in
$($\ref{it:p-adic-unif-a}$)$ can be chosen such that, for either
$\sharp= \widetilde{}$ and $\natural=''$, or $\sharp='$ and
$\natural=''$, we have a commutative diagram
\[\xymatrix{ X^\sharp_{U^{\sharp, p}} \ar[r]\ar[d] &  \bar{G}^\sharp(\BQ) \backslash
(\CH_{\widehat{F^\ur_\fp}}\times G^\sharp(\BA_f)/U^\sharp ) \ar[d] \\
X^\natural_{U^{\natural, p}} \ar[r]  &  \bar{G}^\natural(\BQ)
\backslash (\CH_{\widehat{F^\ur_\fp}}\times
G^\natural(\BA_f)/U^\natural)  }\] compatible with the {
$Z^\sharp(\BQ)\backslash Z^\sharp(\BA_f)/ (Z^\sharp(\BA_f) \cap
U^\sharp)$}-actions on the upper and the \\ {
$Z^\natural(\BQ)\backslash Z^\natural(\BA_f)/(Z^\natural(\BA_f) \cap
U^\natural)$}-actions on the lower, where the left vertical arrow is
induced from the morphism $M^\sharp\rightarrow M^\natural$, and the
right vertical arrow is induced by the identity morphism
$\CH_{\widehat{F_\fp^\ur}} \rightarrow \CH_{\widehat {F_\fp^\ur } }$
and the homomorphism $\alpha: \widetilde{G}=G \times T_E\rightarrow
G''$ or the inclusion $G'\hookrightarrow G''$. Here, in the case of
$\sharp = \widetilde{}$ and $\natural=''$,
$U^\natural=\alpha(U^\sharp)$; in the case of $\sharp='$ and
$\natural=''$, $U^\sharp= U^\natural \cap G'(\BA_f)$.
\end{enumerate}
\end{prop}

The conclusions of Proposition \ref{prop:p-adic-unif} especially
(\ref{it:p-adic-unif-a}) are well known \cite{RZ,V}. However, the
author has no reference for (\ref{it:p-adic-unif-b}), so we provide
some detail of the proof.
\begin{proof}
Assertion (\ref{it:p-adic-unif-a}) in the case of  $\natural='$
comes from \cite[Theorem 6.50]{RZ}.

For the case of $\sharp='$ and $\natural=''$ we put
$$C=Z'(\BQ)\backslash Z'(\BA_f)/ (Z'(\BA_f) \cap U')$$ and
$$Z=Z''(\BQ)\backslash Z''(\BA_f)/ (Z''(\BA_f) \cap U'').$$ Then
$X''_{U''^{p}}$ is $Z$-equivariantly isomorphic to $X'_{U'^p}*_CZ$,
and $\bar{G}''(\BQ) \backslash (\CH_{\widehat{F^\ur_\fp}}\times
G''(\BA_f)/U'' )$ is $Z$-equivariantly isomorphic to
$\Big(\bar{G}'(\BQ) \backslash (\CH_{\widehat{F^\ur_\fp}}\times
G'(\BA_f)/U' )\Big)*_CZ$. So (\ref{it:p-adic-unif-a}) in the case of
$\natural=''$ and (\ref{it:p-adic-unif-b}) in the case of
$\sharp='$, $\natural=''$ follow.

Now we consider the rest cases. Let $H$ be the kernel of the
homomorphism $\alpha: \widetilde{G}=G\times T_E\rightarrow G''$. Put
$$ C = H(\BQ) \backslash H(\BA_f)/(H(\BA_f)\cap \widetilde{U}), \ \
 Z=\widetilde{Z}(\BQ) \backslash
\widetilde{Z}(\BA_f)/(\widetilde{Z}(\BA_f)\cap \widetilde{U}). $$
Put $X_1=\widetilde{X}_{\widetilde{U}^{p}}$ and
$X_2=\widetilde{\bar{G}}(\BQ) \backslash
(\CH_{\widehat{F^\ur_\fp}}\times
\widetilde{G}(\BA_f)/\widetilde{U}_{p,0}\widetilde{U}^p )$. By Lemma
\ref{lem:geom-conn} (\ref{it:geom-conn}), all connected components
of $X_1$ are geometrically connected; it is obvious that all
connected components of $X_2$ are geometrically connected. Thus $Z$
acts freely on the set of components of $X_1$ (resp. $X_2$).
Furthermore $X_1/C$ is isomorphic to
$X''_{\alpha(\widetilde{U}^p)}$, and $X_2/C$ is isomorphic to
$\bar{G}''(\BQ) \backslash (\CH_{\widehat{F^\ur_\fp}}\times
G''(\BA_f)/U''_{p,0}\alpha(\widetilde{U}^p))$.
We have already proved that $X_1/C$ is $Z/C$-equivariantly
isomorphic to $X_2/C$. Applying Lemma \ref{lem:prelim} we obtain
(\ref{it:p-adic-unif-a}) in the case of $\natural=\widetilde{}\ $
and (\ref{it:p-adic-unif-b}) in the case of $\sharp=\widetilde{}\ ,
\natural=''$.
\end{proof}

\begin{rem} By \cite{V} the similar conclusion of Proposition \ref{prop:p-adic-unif} (\ref{it:p-adic-unif-a}) holds for the case of
$\natural=\emptyset$. We use $X_{U^p}$ to denote
$\bar{G}(\BQ)\backslash (\CH_{\widehat{F_\fp^\ur}} \times
G(\BA_f)/U_{p,0}U^p)$, where the action of $\bar{G}(\BQ)$ on
$\CH_{\widehat{F_\fp^\ur}} \times G(\BA_f)/U_{p,0}U^p$ is defined
similarly.
\end{rem}

\section{Local systems and the associated filtered
$\varphi_q$-isocrystals on Shimura Curves} \label{sec:comp-iso}

\subsection{Local systems on Shimura curves} \label{ss:local-sys}

We choose a number field $L$ splitting $F$ and $B$. We identify
$\{\tau_i: F\rightarrow L\}$ with $I=\{\tau_i: F\rightarrow \BC\}$
by the inclusion $L\rightarrow \BC$. Fix an isomorphism
$L\otimes_{\BQ}B= M(2,L)^I$. Then we have a natural inclusion
$G(\BQ)\hookrightarrow \GL(2,L)^I$. Let $\fP$ be a place of $L$
above $\fp$.

For a multiweight $\mathrm{k}=(k_1,\cdots, k_g,w)$ with $k_1\equiv
\cdots k_g\equiv w \mod 2$ and $k_1\geq 2, \cdots, k_g\geq 2$, we
define the morphism $\rho^{(\mathrm{k})}:G\rightarrow \GL(n,L)$
$(n=\prod_{i=1}^g(k_i-1))$ to be the product $\otimes_{i\in
I}[(\Sym^{k_i-2}\otimes \det^{(w-k_i)/2})\circ \check{pr}_i]$. Here
$\check{pr}_i$ denotes the contragradient representation of the
$i$th projection $pr_i: \GL(2,L)^I\rightarrow \GL(2,L)$. The
algebraic group denoted by $G^c$ in \cite[Chapter III]{Mil}  is the
quotient of $G$ by $\ker(\mathrm{N}_{F/\BQ}:F^\times\rightarrow
\BQ^\times)$. As the restriction of $\rho^{(\mathrm{k})}$ to the
center $F^\times$ is the scalar multiplication by
$\mathrm{N}_{F/\BQ}^{-(w-2)}(\cdot)$, $\rho^{(\mathrm{k})}$ factors
through $G^c$, so we can attach to the representation
$\rho^{(\mathrm{k})}$ a $G(\BA_f)$-equivariant smooth $L_\fP$-sheaf
$\CF{(\mathrm{k})}$ on $M$.

Let $p_2: G''_{E_0}\rightarrow G_{E_0}$ be the map induced by the
second projection on $(D\otimes_\BQ E_0)^\times =D^\times \times
D^\times$ corresponding to the conjugate $E_0\rightarrow E_0$. As
the algebraic representation
$\rho''^{(\mathrm{k})}=\rho^{(\mathrm{k})}\circ p_2$ factors through
$G''^c$, we can attach to it a $G''(\BA_f)$-equivariant smooth
$L_\fP$-sheaf $\CF'' (\mathrm{k})$ on $M''$. Let $\CF'
 (\mathrm{k}) $ be the restriction of $\CF'' {(\mathrm{k})}$ to $M'$.

We define a charcater $\bar{\chi}: T_0\rightarrow \BG_m$ such that
on $\BC$-valued point $\bar{\chi}$ is the inverse of the second
projection $T_{0 \BC}=\BC^\times\times \BC^\times\rightarrow
\BC^\times$. Let $\CF(\bar{\chi})$ be the $L_\fP$-sheaf attached to
the representation $\bar{\chi}$. By \cite{Saito} one has the
following $G(\BA_f)\times T(\BA_f)$-equivariant isomorphism of
$L_\fP$-sheaves
\begin{equation} \label{eq:iso-local-sys} \pr_1^* \CF {(\mathrm{k})}
\simeq \alpha^* \CF'' {(\mathrm{k})}\otimes \beta^*
\CF(\bar{\chi}^{-1})^{\otimes (g-1)(w-2)}  \end{equation} on
$M\times N$, where $\pr_1$ is the projection $M\times N\rightarrow
M$.

Note that $L\otimes_{\BQ}D\simeq (\mathrm{M}_2(L)\times
\mathrm{M}_2(L))^I$. For each $i\in I$, the first component
$\mathrm{M}_2(L)$ corresponds to the embedding $E_0\subset L\subset
\BC$ and the second $\mathrm{M}_2(L)$ to its conjugate.  Let $\CF'$
be the local system $R^1g_*L_\fP$ where $g:\CA\rightarrow M'$ is the
universal $\scrO_D$-abelian scheme; it is a sheaf of
$L\otimes_{\BQ}D$-modules. For each $i\in I$, let $e_i\in
L\otimes_{\BQ}D$ be the idempotent whose $(2,i)$-th component is a
rank one idempotent e.g. $\wvec{1}{0}{0}{0}$ and the other
components are zero. Let $\CF'_i$ denote the $e_i$-part $e_i\cdot
R^1g_* L_\fP$. Note that $\CF'_i$ does not depend on the choice of
the rank one idempotent.  By \cite{Saito} we have an isomorphism of
local systems
\begin{equation} \label{eq:isocrystal-general-case}
\CF'{(\mathrm{k})}=\bigotimes_{i\in
I}\Big(\Sym^{k_i-2}\CF'_i\bigotimes (\det
\CF'_i)^{(w-k_i)/2}\Big).\end{equation}

We can define more local systems on $M'$. For
$(\mathrm{k},\mathrm{v})=(k_1,\cdots, k_g;v_1,\cdots, v_g)$, let
$\CF' {(\mathrm{k},\mathrm{v})}$ be the local system
$\bigotimes_{i\in I}\Big(\Sym^{k_i-2}\CF'_i\bigotimes (\det
\CF'_i)^{v_i}\Big)$.


\subsection{Filtered $\varphi_q$-isocrystals associated to
the local systems} \label{ss:compute}

We use $\tilde{\mathrm{k}}$ uniformly to denote
$(\mathrm{k},\mathrm{v})=(k_1,\cdots, k_g;v_1,\cdots, v_g)$ (resp.
$\mathrm{k}=(k_1,\cdots, k_g,w)$) in the case of $\natural='$ (resp.
$\natural=\emptyset,',''$).

We shall need the filtered $\varphi_q$-isocrystal attached to $\CF
{(\tilde{\mathrm{k}})}$. However we do not know how to compute it.
Instead, we compute that attached to $\pr_1^*\CF
{(\tilde{\mathrm{k}})}$. As a middle step we determine the filtered
$\varphi_q$-isocrystals associated to  $\CF' {(\tilde{\mathrm{k}})}$
and  $\CF'' {(\tilde{\mathrm{k}})}$.

For any integers $k$ and $v$ with $k\geq 2$, and any inclusion
$\sigma: F_\fp\rightarrow L_\fP$, let $V_{\sigma}(k,v)$ be the space
of homogeneous polynomials in two variables $X_\sigma$ and
$Y_\sigma$ of degree $k-2$ with coefficients in $L_\fP$; let
$\GL(2,F_\fp)$ act on $V_{\sigma}(k,v)$ by
$$ \wvec{a}{b}{c}{d}P(X_\sigma, Y_\sigma)  = \sigma(ad-bc)^v
P(\sigma(a)X_\sigma+\sigma(c)Y_\sigma,
\sigma(b)X_\sigma+\sigma(d)Y_\sigma).$$ For
$(\mathrm{k},\mathrm{v})=(k_1,\cdots, k_g; v_1,\cdots, v_g)$ we put
$$V(\mathrm{k}, \mathrm{v})=\bigotimes_{\sigma\in
I}V_{\sigma}(k_\sigma, v_\sigma)$$ where the tensor product is taken
over $L_\fP$.

Let $\bar{G}^\natural$ ($\natural= \emptyset, ', '', \widetilde{}\
$) be the groups defined in Section \ref{ss:notations}. For
$\natural= \widetilde{}$, via the projection
$\bar{G}^{\natural}(\BQ_p)\rightarrow \GL(2,F_\fp)$,
$V(\tilde{\mathrm{k}})$ becomes a $\bar{G}^\natural(\BQ_p)$-module.
For $\natural=',''$, via the projection of
$\bar{G}^\natural(\BQ_p)\subset \mathrm{GL}(2, F_\fp) \times
\mathrm{GL}(2, F_\fp)$ to the second factor, $V(\tilde{\mathrm{k}})$
becomes a $\bar{G}^\natural(\BQ_p)$-module. In each case via the
inclusion $\bar{G}^\natural(\BQ)\hookrightarrow
\bar{G}^\natural(\BQ_p)$, $V(\tilde{\mathrm{k}})$ becomes a
$\bar{G}^\natural(\BQ)$-module. Using the $p$-adic uniformization of
$X^\natural=X^\natural_{U^{\natural,p}}$ we attach to this
$\bar{G}^\natural(\BQ)$-module a local system
$\CV^\natural(\tilde{\mathrm{k}})$ on $X^\natural$.

Let $\varphi_{q,\mathrm{k},\mathrm{v}}$ be the operator on $
V(\mathrm{k}, \mathrm{v}) $  $$\bigotimes_\sigma P_\sigma(X_\sigma,
Y_\sigma) \mapsto \prod_\sigma \sigma(-\pi)^{v_\sigma} \cdot
\bigotimes_\sigma P_\sigma(Y_\sigma, \sigma(\pi)X_\sigma ).$$  For
$\mathrm{k}=(k_1,\cdots, k_g, w)$ we put
$$ V(\mathrm{k}) = V(k_1,\cdots, k_g; (w-k_1)/2,\cdots,
(w-k_g)/2)  $$ and $$ \varphi_{q, \mathrm{k}}=
\varphi_{q,(k_1,\cdots, k_g; (w-k_1)/2,\cdots, (w-k_g)/2)}.$$

Let $\mathscr{F}^\natural(\tilde{\mathrm{k}})$ be the filtered
$\varphi_q$-isocrystal $
\CV^\natural(\tilde{\mathrm{k}})\otimes_{\BQ_p} \scrO_{X^\natural}$
on $X^\natural$ with the $q$-Frobenius $\varphi_{q,
\tilde{\mathrm{k}} }\otimes \varphi_{q,\scrO_{X^\natural}}$ and the
connection $1\otimes \mathrm{d}:
\CV^\natural(\tilde{\mathrm{k}})\otimes_{\BQ_p}
\scrO_{X^\natural}\rightarrow
\CV^\natural(\tilde{\mathrm{k}})\otimes_{\BQ_p}
\Omega^1_{X^\natural}$; the filtration on
\begin{eqnarray}\label{eq:decom-isocrystal}   \CV^\natural(\tilde{\mathrm{k}})\otimes_{\BQ_p}
\scrO_{X^\natural} = \bigoplus_{\tau:F_\fp\hookrightarrow L_\fP}
\CV^\natural(\tilde{\mathrm{k}})\otimes_{\tau, F_\fp}
\scrO_{X^\natural} \end{eqnarray}
is given by
\begin{eqnarray*} && \Fil^{j+v_\tau}
(\CV^\natural(\tilde{\mathrm{k}})\otimes_{\tau, F_\fp} \scrO_{X^\natural})  \\
&=& \left\{\begin{array}{ll}
\CV^\natural(\tilde{\mathrm{k}})\otimes_{\tau, F_\fp}
\scrO_{X^\natural} & \text{ if } j\leq 0 ,
\\ \text{the } \scrO_{X^\natural}\text{-submodule locally generated by polynomials}  & \\
\hskip 110pt \text{in } V(\tilde{\mathrm{k}}) \text{ divided by } (z
X_\tau +  Y_\tau)^{j}   & \text{ if } 1\leq j\leq k_\tau-2
\\ 0 & \text{ if } j \geq k_\tau - 1 \end{array}\right.
\end{eqnarray*}
with the convention that $v_\tau=\frac{w-k_\tau}{2}$ in the case of
$\tilde{\mathrm{k}}=(k_1,\cdots, k_g,w)$, where $z$ is the canonical
coordinate on $\CH_{\widehat{F_\fp^\ur}}$.

\begin{lem}\label{lem:isocrystal-special}
When $k_1=\cdots=k_{i-1}=k_{i+1}=\cdots=k_g=2$, $k_i=3$, and
$v_1=\cdots=v_g=0$, the filtered $\varphi_q$-isocrystal attached to
$\CF' {(\mathrm{k},\mathrm{v})}$ is isomorphic to
$\mathscr{F}'(\mathrm{k},\mathrm{v})$.
\end{lem}
\begin{proof}
Let $\tilde{e}_i\in L\otimes_{\BQ}D$ be the idempotent whose
$(2,i)$-th component is $\wvec{1}{0}{0}{1}$ and the other component
are zero. Let $\CA$ be the universal $\scrO_D$-abelian scheme over
$M'$, $\widehat{\CA}$ the formal module on $X'$ attached to $\CA$.
Note that $\tilde{e}_i
(\mathfrak{o}_{L_\fP}\otimes_{\BZ_p}\widehat{\CA})$ is just the
pullback of $\mathfrak{o}_{L_\fP}\otimes_{\tau_i,
\mathfrak{o}_{F_\fp} }\CG$ by the projection $X'_{U'^p}\rightarrow
(\bar{G}'(\BQ)\cap U'^pU'_{p,0})\backslash
\CH_{\widehat{F^\mathrm{ur}_\fp}}$ \cite[6.43]{RZ}, where $\CG$ is
the universal special formal $\scrO_{B_\fp}$-module (forgetting the
information of $\rho$ in Drinfeld's moduli problem).

As $L_\fP$ splits $B_\fp$, $L_\fP$ contains all embeddings of
$F^{(2)}_\fp$. The embedding $\tau_i: F_\fp\hookrightarrow L_\fP$
extends in two ways to $F^{(2)}_\fp$ denoted respectively by
$\tau_{i,0}$ and $\tau_{i,1}$. Then
$$ \fo_{ L_\fP } \otimes_{ \tau_i, \fo_{F_\fp} } \fo_{B_\fp}= \fo_{L_\fP} \otimes_{ \tau_{i,0},
\fo_{F_\fp^{(2)}} } \fo_{B_{\fp}} \bigoplus \fo_{L_\fP} \otimes_{
\tau_{i,1}, \fo_{ F_\fp^{(2)} } } \fo_{B_\fp}. $$ We decompose
$\mathfrak{o}_{L_\fP}\otimes_{\tau_i, \mathfrak{o}_{F_\fp}}\CG$ into
the sum of two direct summands according to the action of
$\fo_{F_\fp^{(2)}}\subset \fo_{B_\fp}$: $ \fo_{F_\fp^{(2)}}$ acts by
 $\tau_{i,0}$ on the first direct summand, and acts by
 $\tau_{i,1}$ on the second. Without loss of generality we
may assume that $e_i$ in the definition of $\CF'_i$ (see Section
\ref{ss:local-sys}) is chosen such that $e_i$ is the projection onto
the first direct summand. So $e_i
(\mathfrak{o}_{L_\fP}\otimes_{\BZ_p}\widehat{\CA})$ is just the
pullback of $\mathfrak{o}_{L_\fP}\otimes_{\tau_{i,0},
\mathfrak{o}_{F_\fp^{(2)}} }\CG$ by the projection
$X'_{U'^p}\rightarrow (\bar{G}'(\BQ)\cap U'^pU'_{p,0})\backslash
\CH_{\widehat{F^\mathrm{ur}_\fp}}$. Now the statement of our lemma
follows from the discussion in Section \ref{ss:univ-fil-phi-mod}.
\end{proof}

\begin{prop}\label{prop:fil-F-iso}
The filtered $\varphi_q$-isocrystal attached to $\CF'
{(\mathrm{k},\mathrm{v})}$ is isomorphic to
$\mathscr{F}'(\mathrm{k},\mathrm{v})$.
\end{prop}
\begin{proof}
Let $\mathscr{F}'_i$ denote the filtered $\varphi_q$-isocrystal
attached to $\mathcal{F}'_i$. By (\ref{eq:isocrystal-general-case})
the filtered $\varphi_q$-isocrystal attached to
$\CF'(\mathrm{k},\mathrm{v})$ is isomorphic to
\begin{equation} \label{eq:isocrystal-general-case}
\bigotimes_{i\in I}\Big(\Sym^{k_i-2}\mathscr{F}'_i\bigotimes (\det
\mathscr{F}'_i)^{(w-k_i)/2}\Big).\end{equation} By Lemma
\ref{lem:isocrystal-special} a simple computation implies our
conclusion.
\end{proof}

\begin{cor} \label{cor:fil-mod-1}
The filtered $\varphi_q$-isocrystal attached to
$\CF''{(\mathrm{k})}$ is isomorphic to $\mathscr{F}''(\mathrm{k})$.
\end{cor}
\begin{proof}
This follows from Proposition \ref{prop:fil-F-iso} and \cite[Lemma
6.1]{Saito}.
\end{proof}

\begin{lem}\label{lem:fil-mod-1} The filtered $\varphi_q$-isocrystal associated to the
local system $\CF(\bar{\chi})$ over
$(N_{E_0,0})_{\widehat{F_\fp^\ur}}$ is $\CF(\bar{\chi})\otimes
\scrO_{(N_{E_0,0})_{\widehat{F_\fp^\ur}}}$ with the $q$-Frobenius
being $1\otimes \varphi_{q, (N_{E_0,0})_{\widehat{F_\fp^\ur}}}$ and
the filtration being trivial.
\end{lem}
\begin{proof}
We only need to show that any geometric point of $(N_{E_0,
0})_{\widehat{F_\fp^\ur}}$ is defined over $\widehat{F_\fp^\ur}$.

Let $h_{E_0}$ be as in Section \ref{ss:shimura-curves}, $\mu$ the
cocharacter of $T_0$ defined over $E_0$ attached to $h_{E_0}$. Let
$r$ be the composition $$\BA_{E_0}^\times\xrightarrow{\mu}
T_0(\BA_{E_0}) \xrightarrow{\mathrm{N}^{E_0}_{\BQ}} T_0(\BA).$$ Let
$$ \mathrm{art}_{E_0}: \BA^\times_{E_0}\twoheadrightarrow
\Gal(E_0^\mathrm{ab}/E_0) $$ be the reciprocal of the reciprocity
map from class field theory. For any compact open subgroup $U$ of
$T_0(\BA_{f})$, $\Gal(\overline{\BQ}/E_0)$ acts on
$(N_{E_0})_U(\overline{\BQ})= T_0(\BQ)\backslash T_0(\BA_{f})/U$ by
$\sigma (T_0(\BQ) a U) = T_0(\BQ) r_f(s_\sigma) a U$, where
$s_\sigma$ is any id\`ele such that
$\mathrm{art}_{E_0}(s_\sigma)=\sigma| E^\mathrm{ab}$, and $r_{f}$ is
the composition
$$\BA^\times_{E_0}\rightarrow T_0(\BA) \rightarrow T_0(\BA_{f})$$ of
$r$ and the projection map $T_0(\BA)\rightarrow T_0(\BA_{f})$. Let
$\mathcal{I}$ be the subgroup of $\Gal(\overline{\BQ}/E_0)$
consisting of $\sigma$ such that $s_\sigma\in r_f^{-1}(U)$. Put
$\mathcal{K}=\overline{\BQ}^{\mathcal{I}}$. Then any geometric point
of $(N_{E_0})_U$ is defined over $\mathcal{K}$. Observe that, when
$U$ is of the form $U_{p,0}U^p$ with $U^p$ a compact open subgroup
of $T_0(\BA_f^p)$ and $U_{p,0}$ the maximal compact open subgroup of
$T_0(\BQ_p)$, $\mathcal{K}$ is unramified over $p$. Therefore, any
geometric point of $(N_{E_0, 0})_{\widehat{F_\fp^\ur}}$ is already
defined over $\widehat{F_\fp^\ur}$.
\end{proof}

\begin{cor} \label{cor:fil-isocry}
The filtered $\varphi_q$-isocrystal attached to
$\pr_1^*\CF(\mathrm{k})$ is $\pr_1^*\mathscr{F}(\mathrm{k})$.
\end{cor}
\begin{proof} By
(\ref{eq:iso-local-sys}) the filtered $\varphi_q$-isocrystal
attached to $\pr_1^*\CF {(\mathrm{k})}$ is the tensor product of the
filtered $\varphi_q$-isocrystal attached to $\alpha^* \CF''
{(\mathrm{k})}$ and that attached to $\beta^*
\CF(\bar{\chi}^{-1})^{(g-1)(w-2)}$. Our conclusion follows from
Proposition \ref{prop:p-adic-unif} (\ref{it:p-adic-unif-b}) (in the
case of $\sharp=\widetilde{}$ and $\natural= ''$), Corollary
\ref{cor:fil-mod-1} and Lemma \ref{lem:fil-mod-1}.
\end{proof}

It is rather possible that the filtered $\varphi_q$-isocrystal
attached to $\CF(\mathrm{k})$ is $\mathscr{F}(\mathrm{k})$. But the
author does not know how to descent the conclusion of Corollary
\ref{cor:fil-isocry} to $X_{U^p}$.

\section{The de Rham cohomology} \label{sec:cover-hodge}

\subsection{Covering filtration and Hodge filtration for de Rham
cohomology} \label{ss:cover-hodge}

We fix an arithmetic Schottky group $\Gamma$ that is cocompact in
$\PGL(2, F_\fp)$. Then $\Gamma$ acts freely on $\CH$, and the
quotient $X_\Gamma=\Gamma\backslash \CH$ is the rigid analytic space
associated with a proper smooth curve over $F_\fp$. Here we write
$\CH$ for $\CH_{\widehat{F_\fp^\ur}}$.

We recall the theory of de Rham cohomology of local systems over
$X_\Gamma$ \cite{SchLoc, SchSt, deSh-1, deSh-2}.

We denote by $\widehat{\CH}$ the canonical formal model of $\CH$.
The curve $X_\Gamma$ has a canonical semistable module
$\CX_\Gamma=\Gamma\backslash \widehat{\CH}$; the special fiber
$\CX_{\Gamma,s}$ of $\CX_\Gamma$ is isomorphic to $\Gamma\backslash
\widehat{\CH}_s$.

The graph $\mathrm{Gr}(\CX_{\Gamma,s})$ is closely related to the
Bruhat-Tits tree $\CT$ for $\PGL(2,F_\fp)$. The group $\Gamma$ acts
freely on the tree $\CT$. Let $\CT_\Gamma$ denote the quotient tree.
The set of connected components of the special fiber
$\CX_{\Gamma,s}$ is in one-to-one correspondence to the set
$\mathrm{V}(\CT_\Gamma)$ of vertices of $\CT_\Gamma$, each component
being isomorphic to the projective line over $k(=$the resider field
of $F_\fp$). Write $\{\RP^1_v\}_{v\in \mathrm{V}(\CT_\Gamma)}$ for
the set of components of $\CX_{\Gamma,s}$. The singular points of
$\CX_{\Gamma,s}$ are ordinary $k$-rational double singular points;
they correspond to (unoriented) edges of $\CT_\Gamma$. Two
components $\RP^1_{u}$ and $\RP^1_v$ intersect if and only if $u$
and $v$ are adjacent; in this case, they intersect at a singular
point. For simplicity we will use the edge $e$ joining $u$ and $v$
to denote this singular point. There is a reduction map from
$X_\Gamma^\an$ to $\CX_{\Gamma,s}$. For a closed subset $U$ of
$\CX_{\Gamma,s}$ let $]U[$ denote the tube of $U$ in $X_\Gamma^\an$.
Then $\{]\RP^1_v[\}_{v\in \mathrm{V}(\CT_\Gamma)}$ is an admissible
covering of $X_\Gamma^\an$. Clearly $]\RP^1_{o(e)}[\cap
]\RP^1_{t(e)}[=]e[$.

Let $L$ be a field that splits $F_\fp$. Fix an embedding $\tau:
F_\fp\hookrightarrow L$. Let $V$ be an $L[\Gamma]$-module that comes
from an algebraic representation of $\PGL(2,F_\fp)$ of the form
$V(\mathrm{k})$ with $\mathrm{k}=(k_1,\cdots, k_g;2)$. We will
regard $V$ as an $F_\fp$-vector space by $\tau$. Let
$\mathscr{V}=\mathscr{V}(\mathrm{k})$ be the local system on
$X_\Gamma$ associated with $V$. Let
$H^*_{\dR,\tau}(X_\Gamma,\mathscr{V})$ be the hypercohomology of the
complex $\mathscr{V}\otimes_{\tau, F_\fp}
\Omega_{X_\Gamma}^\bullet$.

We consider the Mayer-Vietorus exact sequence attached to
$H^*_{\dR,\tau}(X_\Gamma, \mathscr{V})$ with respect to the
admissible covering $\{]\RP^1_v[\}_{v\in \mathrm{V}(\CT_\Gamma)}$.
As a result we obtain an injective map
$$ \iota: (\bigoplus_{e\in \mathrm{E}(\CT_\Gamma)}H^0_{\dR,\tau}(]e[,\mathscr{V}))^-/
\bigoplus_{v\in
\mathrm{V}(\CT_\Gamma)}H^0_{\dR,\tau}(]\RP^1_v[,\mathscr{V})
\hookrightarrow  H^1_{\dR,\tau}(X_\Gamma^\an, \mathscr{V}) .$$

As $]\RP^1_v[$ and $]e[$ are quasi-Stein, a simple computation shows
that $H^0_{\dR,\tau}(]\RP^1_v[,\mathscr{V})$ and
$H^0_{\dR,\tau}(]e[,\mathscr{V})$ are isomorphic to $V$. Let
$C^0(V)$ be the space of $V$-valued functions on $\mathrm{V}(\CT)$,
$C^1(V)$ the space of $V$-valued functions on $\mathrm{E}(\CT)$ such
that $f(e)=-f(\bar{e})$. Let $\Gamma$ act on $C^i(V)$ as $f\mapsto
\gamma \circ f\circ \gamma^{-1}$. Then we have a
$\Gamma$-equivariant short exact sequence
\begin{equation} \label{eq:cover-sq}
\xymatrix{ 0\ar[r] & V \ar[r] & C^0(V) \ar[r]^{\partial} & C^1(V)
\ar[r] & 0 }\end{equation} where $\partial(f)(e)=f(o(e))-f(t(e))$.
Observe that \begin{eqnarray*} \bigoplus_{v\in
\mathrm{V}(\CT_\Gamma)}H^0_{\dR,\tau}(]\RP^1_v[,\mathscr{V}) & \cong
& C^0(V)^\Gamma , \\ (\bigoplus_{e\in
\mathrm{E}(\CT_\Gamma)}H^0_{\dR,\tau}(]e[,\mathscr{V}))^- & \cong &
C^1(V)^\Gamma  \end{eqnarray*} and the map $$\bigoplus_{v\in
\mathrm{V}(\CT_\Gamma)}H^0_{\dR,\tau}(]\RP^1_v[,\mathscr{V})\rightarrow
(\bigoplus_{e\in
\mathrm{E}(\CT_\Gamma)}H^0_{\dR,\tau}(]e[,\mathscr{V}))^-$$
coincides with $\partial$. Thus $$(\bigoplus_{e\in
\mathrm{E}(\CT_\Gamma)}H^0_{\dR,\tau}(]e[,\mathscr{V}))^-/
\bigoplus_{v\in
\mathrm{V}(\CT_\Gamma)}H^0_{\dR,\tau}(]\RP^1_v[,\mathscr{V}) $$ is
isomorphic to $C^1(V)^\Gamma/\partial C^0(V)^\Gamma$. From
(\ref{eq:cover-sq}) we get the injective map $$\delta:
C^1(V)^\Gamma/\partial C^0(V)^\Gamma \hookrightarrow
H^1(\Gamma,V).$$

Let $C^1_\har(V)$ be the space of harmonic forms $$
C^1_\har(V):=\{f: \Edg(\CT)\rightarrow V | f(e)=-f(\bar{e}), \
\forall\ v, \sum_{t(e)=v} f(e)=0  \}.$$ Fixing some $v\in
\mathrm{V}(\CT)$, let $\epsilon$ be the map
$C^1_\har(V)^\Gamma\rightarrow H^1(\Gamma, V)$ \cite[(2.26)]{CMP}
defined by \begin{equation} \label{eq:sch}
c\mapsto (\gamma \mapsto
\sum_{e:v\rightarrow \gamma v}c(e)), \end{equation} where the sum
runs over the edges joining $v$ and $\gamma v$; $\epsilon$ does not
depend on the choice of $v$. By \cite[Appendix A]{CMP} $\epsilon$ is
minus the composition
$$C^1_\har(V)^\Gamma\rightarrow C^1(V)^\Gamma/\partial
C^0(V)^\Gamma\xrightarrow{\delta} H^1(\Gamma,V) ,$$ and is an
isomorphism. Combining this with the injectivity of $\delta$ we
obtain that both the natural map $C^1_\har(V)^\Gamma\rightarrow
C^1(V)^\Gamma/\partial C^0(V)^\Gamma$ and $\delta$ are isomorphisms.
Below, we will identify $C^1_\har(V)^\Gamma$ with
$C^1(V)^\Gamma/\partial C^0(V)^\Gamma$.

By \cite{deSh-2} we have
\begin{equation}\label{eq:second-diff-form}\begin{aligned}
H^1_{\dR,\tau}(X_\Gamma, \mathscr{V}) \cong \{ V &\text{-valued
differentials of second kind on }X_\Gamma\} / \\ & \{d  f | f \text{
a } V\text{-valued meromorphic function on }X_\Gamma\}
.\end{aligned}\end{equation} In loc. cit, de Shalit only considered
a special case, but his argument is valued for our general case. If
$\omega$ is a $\Gamma$-invariant $V$-valued differential of the
second kind on $\CH$, let $F_\omega$ be a primitive of it
\cite{deSh-1}, which is defined by Coleman's integral \cite{Col}.
\footnote{Precisely we choose a branch of Coleman's integral.}  Let
$P$ be the map
$$P: H^1_{\dR,\tau}(X_\Gamma, \mathscr{V})\rightarrow H^1(\Gamma, V) ,
\hskip 10pt \omega \mapsto (\gamma \mapsto
\gamma(F_\omega)-F_\omega).$$ Note that $P\circ \iota$ coincides
with $\delta$. Thus $P$ splits the inclusion $\iota\circ
\delta^{-1}:H^1(\Gamma,V)\rightarrow
H^1_{\dR,\tau}(X_\Gamma,\mathscr{V})$.

Let $I$ be the map $$ I: H^1_{\dR,\tau}(X_\Gamma,
\mathscr{V})\rightarrow C^1_\har(V)^\Gamma, \omega\mapsto (e\mapsto
\Res_e(\omega)).$$

Now, we suppose that $\Gamma$ is of the form in \cite[Appendix
A]{CMP}. We do not describe it precisely, but only point out that
$\Gamma_{i, 0}$ in Section \ref{ss:apply} is of this form.

\begin{prop}\label{prop:deShalit} We have an exact sequence called
the covering filtration exact sequence
$$ \xymatrix{ 0 \ar[r] & H^1(\Gamma, V)\ar[r]^{\iota\circ \delta^{-1}} &
H^1_{\dR,\tau}(X_\Gamma, \mathscr{V}) \ar[r]^{I} &
C^1_\har(V)^\Gamma\ar[r] & 0 . }
$$
\end{prop}
\begin{proof} What we need to prove is that the map $$H^1_{\dR,\tau}(X_\Gamma,
\mathscr{V})\rightarrow H^1(\Gamma, V)\oplus
C^1_\har(V)^\Gamma\hskip 10pt \omega\mapsto (P(\omega), I(\omega))$$
is an isomorphism. When $V$ is the trivial module, this is already
proved in \cite{deSh-2}. So we assume that $V$ is not the trivial
module. First we prove the injectivity of the above map. For this we
only need to repeat the argument in \cite[Theorem 1.6]{deSh-2}. Let
$\omega$ be a $\Gamma$-invariant differential form of second kind on
$\CH$ such that $P([\omega])=I([\omega])=0$, where $[\omega]$
denotes the class of $\omega$ in $H^1_{\dR,\tau}(X_\Gamma,
\mathscr{V})$. Let $F_\omega$ be a primitive of $\omega$. As
$I(\omega)=0$, the residues of $\omega$ vanish, and thus $F_\omega$
is meromorphic. As $P(\omega)=0$, we may adjust $F_\omega$ by a
constant so that it is $\Gamma$-invariant. By
(\ref{eq:second-diff-form}) we have $[\omega]=0$. To show the
surjectivity we only need to compare the dimensions. By
\cite[Appendix A]{CMP} we have
$$\dim_{F_\fp}C^1_\har(V)^\Gamma=\dim_{F_\fp} H^1(\Gamma, V)$$ and
$$\dim_{F_\fp}H^1(\Gamma,V)=\dim_{F_\fp}H^1(\Gamma, V^*)$$ where
$V^*=\Hom_{F_\fp}(V, F_\fp)$ is the dual $F_\fp[\Gamma]$-module. By
\cite[Theorem 1]{SchLoc}  we have
$$\dim_{F_\fp}H^1_{\dR,\tau}(X_\Gamma,
\mathscr{V})=\dim_{F_\fp}H^1(\Gamma, V)+\dim_{F_\fp}H^1(\Gamma,
V^*).$$ Hence $$ \dim_{F_\fp}H^1_{\dR,\tau}(X_\Gamma,
\mathscr{V})=\dim_{F_\fp}C^1_\har(V)^\Gamma+\dim_{F_\fp} H^1(\Gamma,
V),
$$ as desired.
\end{proof}

We have also a Hodge filtration exact sequence
$$ \xymatrix{ 0 \ar[r] & H^0(X_\Gamma,  \mathscr{V}\otimes_{F_\fp,\tau}\Omega^1_{X_\Gamma}) \ar[r] & H^1_{\dR,\tau}(X_\Gamma, \mathscr{V})
\ar[r] & H^1(X_\Gamma,
\mathscr{V}\otimes_{F_\fp,\tau}\CO_{X_\Gamma})\ar[r] & 0. }
$$ This exact sequence and the covering filtration exact sequence
fit into the commutative diagram
\[
\xymatrix{ && 0\ar[d] &&
\\ && H^1(\Gamma, V)\ar[d]^{\iota\circ \delta^{-1}}\ar[rd]^{\simeq} && \\ 0 \ar[r] &
H^0(X_\Gamma,
\mathscr{V}\otimes_{F_\fp,\tau}\Omega^1_{X_\Gamma})\ar[rd]^{\simeq}
\ar[r] &
 H^1_{\dR,\tau}(X_\Gamma, \mathscr{V}) \ar[r]\ar[d]^{I} &
 H^1(X_\Gamma, \mathscr{V}\otimes_{F_\fp,\tau}\CO_{X_\Gamma}) \ar[r] & 0 .\\
 &&  C^1_\har(V)^\Gamma \ar[d] && \\ && 0 &&  } \]
The diagonal arrows are isomorphisms. Indeed, by \cite[Section
3]{deSh-1} the south-west arrow
is an isomorphism; 
one easily deduces from this that the north-east arrow is also an
isomorphism. In particular, we have a Hodge-like decomposition
\begin{equation} \label{eq:hodge-cover}
H^1_{\dR,\tau}(X_\Gamma, \mathscr{V}) = H^0(X_\Gamma,
\mathscr{V}\otimes_{F_\fp,\tau}\Omega^1_{X_\Gamma}) \bigoplus
\iota\circ \delta^{-1} ( H^1(\Gamma, V) ). \end{equation}

\subsection{De Rham cohomology of $\mathscr{F}(\mathrm{k})$}
\label{ss:apply}

Let $\bar{B}$ be as in Section \ref{ss:notations}. Write
$\widehat{\bar{B}}^\times :=(\bar{B}\bigotimes \BA_{ f})^\times$ and
$\widehat{\bar{B}}^{p,\times} :=(\bar{B}\bigotimes \BA_{
f}^p)^\times$. Let $U$ be a compact open subgroup of $G(\BA_f)$ that
is of the form $U_{p,0}U^p$. We identify $U^p=\prod\limits_{\fl\neq
\fp}U_\fl$ with a subgroup of $\widehat{\bar{B}}^{p,\times}$.

Write $\widehat{\bar{B}}^{p,\times} = \sqcup_{i=1}^{h}
\bar{B}^\times x_i U^p$. For $i=1,\cdots, h$, put
$$ \widetilde{\Gamma}_i = \{ \gamma\in \bar{B}^\times:  \gamma_\mathfrak{l}\in (x_i)_\mathfrak{l} U_\mathfrak{l} (x_i)_\mathfrak{l}^{-1}
\text{ for }\mathfrak{l}\neq \mathfrak{p} \}  .$$  Then $X_{U^p}$ is
isomorphic to
$$ \bar{B}^\times \backslash (\CH \times G(\BQ_p)/U_{p,0} \times \widehat{\bar{B}}^{p,\times}/U^p)
\cong \sqcup_{i=1}^h \widetilde{\Gamma}_i \backslash (\CH \times
\BZ) .
$$ Here we identify $\BZ$ with $G(\BQ_p)/U_{p,0}$. Note that
$\widetilde{\Gamma}_i$ acts transitively on $G(\BQ_p)/U_{p,0}$.

Note that, for every point in $\BZ=G(\BQ_p)/U_{p,0}$ it is fixed by
$\gamma\in \bar{B}^\times$ if and only if $\gamma_\fp$ is in
$\mathrm{GL}(2, \scrO_{F_\fp})$. Put \begin{eqnarray*}
\widetilde{\Gamma}_{i,0} &=& \{ \gamma\in \widetilde{\Gamma}_i:
|\det(\gamma_\mathfrak{p})|_\fp=1  \}
\\ &=& \{ \gamma\in \bar{B}^\times:  \gamma_\mathfrak{l}\in
(x_i)_\mathfrak{l} U_\mathfrak{l} (x_i)_\mathfrak{l}^{-1} \text{ for
}\mathfrak{l}\neq \mathfrak{p} \text{ and }
|\det(\gamma_\mathfrak{p})|_\fp=1  \} .
\end{eqnarray*} Let $\Gamma_{i,0}$
be the image of $\widetilde{\Gamma}_{i,0}$ in $\mathrm{PGL}(2,
F_\fp)$. Then we have an isomorphism \begin{equation}
\label{eq:decomp} X_{U^p} \cong \sqcup_{i=1}^h
 \Gamma_{i,0}\backslash \CH .\end{equation}

Applying the constructions in Section \ref{ss:cover-hodge} to each
part $\Gamma_{i,0}\backslash \CH$ of $X_{U^p}$, we obtain operators
$\iota$, $P$ and $I$.

\section{Automorphic Forms on totally definite quaternion algebras and Teitelbaum type
$L$-invariants} \label{sec:Teit-L-inv}

In this section we recall Chida, Mok and Park's definition of
Teitelbaum type $L$-invariant \cite{CMP}.

\subsection{Automorphic Forms on totally definite quaternion algebras}

We recall the theory of automorphic forms on totally definite
quaternion algebras.

Let $\bar{B}$ be as in Section \ref{sec:p-unif}, which is a totally
definite quaternion algebra over $F$. Let
$\Sigma=\prod_{\mathfrak{l}}\Sigma_{\mathfrak{l}}$ be a compact open
subgroup of $\widehat{\bar{B}}^\times$.

Let $\chi_{F,\cyc}:\BA_F^\times/F^\times\rightarrow \BZ_p^\times$ be
the Hecke character obtained by composing the cyclotomic character
$\chi_{\BQ,\cyc}: \BA_\BQ^\times/\BQ^\times\rightarrow \BZ_p^\times$
and the norm map from $\BA_F^\times$ to $\BA_\BQ^\times$.

\begin{defn} An {\it automorphic form} on $\bar{B}^\times$, of {\it
weight} $\mathrm{k}=(k_1,\cdots, k_g, w)$ and {\it level} $\Sigma$,
is a function $\mathbf{f}: \widehat{\bar{B}}^\times\rightarrow
V(\mathrm{k})$ that satisfies
$$ \mathbf{f}( z \gamma b u ) = \chi_{F,\cyc}^{2-w}(z) (u_p^{-1}\cdot \mathbf{f}(b)) $$
for all $\gamma\in \bar{B}^\times$, $u\in \Sigma$, $b\in
\widehat{\bar{B}}^\times$  and $z\in \widehat{F}^\times$. Denote by
$S^{\bar{B}}_{\mathrm{k}}(\Sigma)$ the space of such forms. Remark
that our $S^{\bar{B}}_{\mathrm{k}}(\Sigma)$ coincides with
$S^{\bar{B}}_{\mathrm{k}',\mathrm{v}}(\Sigma)$ for
$\mathrm{k}'=(k_1-2,\cdots, k_g-2)$ and
$\mathrm{v}=(\frac{w-k_1}{2}, \frac{w-k_2}{2}, \cdots,
\frac{w-k_g}{2})$ in \cite{CMP}.
\end{defn}

Observe that a form $\mathbf{f}$ of level $\Sigma$ is determined by
its values on the finite set $\bar{B}^\times\backslash
\widehat{\bar{B}}^\times/ \Sigma$. As in Section \ref{ss:apply}
write $\widehat{\bar{B}}^{\times} = \sqcup_{i=1}^{h} \bar{B}^\times
x_i \GL(2, F_\fp) \Sigma$; for $i=1,\cdots, h$, put
$$ \widetilde{\Gamma}_i = \{ \gamma\in \bar{B}^\times:  \gamma_\mathfrak{l}\in (x_i)_\mathfrak{l} \Sigma_\mathfrak{l} (x_i)_\mathfrak{l}^{-1}
\text{ for }\mathfrak{l}\neq \mathfrak{p} \}  .$$ Then we have a
bijection
$$ \sqcup_{i=1}^h \widetilde{\Gamma}_i\backslash \GL(2,F_\fp)/ \Sigma_\fp \xrightarrow{\sim} \bar{B}^\times \backslash \widehat{\bar{B}}^\times/\Sigma .
$$ The class of $g$ in $\widetilde{\Gamma}_i\backslash \GL(2,F_\fp)/ \Sigma_\fp$ corresponds to the class of $x_{i,\fp} g_\fp $ in
$\bar{B}^\times \backslash \widehat{\bar{B}}^\times/\Sigma$, where
 $g_\fp$ is the element of
$\widehat{\bar{B}}^\times$ that is equal to $g$ at the place $\fp$,
and equal to the identity at each other place. Using this we can
attach to an automorphic form $\mathbf{f}$ of weight $\mathrm{k}$
and level $\Sigma$ an $h$-tuple of functions $(f_1, \cdots, f_h)$ on
$\GL(2, F_\fp)$ with values in $V(\mathrm{k})$ defined by
$f_i(g)=\mathbf{f}(x_{i,\fp} g_\fp)$. The function $f_i$ satisfies
$$ f_i(\gamma_\fp g uz ) = \chi_{F,\cyc}^{2-w}(z) u^{-1}\cdot f_i(g)
$$ for all $\gamma_\fp \in \widetilde{\Gamma}_i$, $g\in
\GL(2,F_\fp)$,  $u\in \Sigma_\fp$ and $z\in F_\fp^\times$.

For each prime $\fl$ of $F$ such that $\bar{B}$ splits at $\fl$,
$\fl\neq \fp$, and $\Sigma_\fl$ is maximal, one define a Hecke
operator $T_\fl$ on $S^{\bar{B}}_{\mathrm{k}}(\Sigma)$ as follows.
Fix an isomorphism $\iota_\fl: B_\fl\rightarrow M_2(F_\fl)$ such
that $\Sigma_\fl$ becomes identified with
$\mathrm{GL}_2(\mathfrak{o}_{F_\fl})$. Let $\pi_\fl$ be a
uniformizer of $\mathfrak{o}_{F_\fl}$. Given a double coset
decomposition
$$ \mathrm{GL}_2(\mathfrak{o}_{F_\fl}) \wvec{1}{0}{0}{\pi_\fl} \mathrm{GL}_2(\mathfrak{o}_{F_\fl})
=\coprod b_i \mathrm{GL}_2(\mathfrak{o}_{F_\fl}) $$ we define the
Hecke operator $T_\fl$ on $S^{\bar{B}}_{\mathrm{k}}(\Sigma)$ by
$$ (T_\fl \mathbf{f})(b) =\sum_i \mathbf{f}(b b_i) . $$ We define $U_\fp$ similarly. Let
$\BT_{\Sigma}$ be the Hecke algebra generated by $U_\fp$ and these
$T_\fl$.

Denote by $\mathfrak{o}_F^{(\fp)}$ the ring of $\fp$-integers of $F$
and $(\mathfrak{o}_F^{(\fp)})^\times$ the group of $\fp$-units of
$F$. We have $\widetilde{\Gamma}_i\cap F^\times =
(\mathfrak{o}_F^{(\fp)})^\times$. For $i=1,\cdots, h$, put
$\Gamma_i=\widetilde{\Gamma}_i/(\mathfrak{o}_F^{(\fp)})^\times$.
Consider the following twisted action of $\widetilde{\Gamma_i}$ on
$V(\mathrm{k})$:
$$ \gamma \star v = |\mathrm{Nrd}_{\bar{B}/F}\gamma|_\fp^{\frac{w-2}{2}} \gamma_\fp \cdot v .$$
Then $(\mathfrak{o}_F^{(\fp)})^\times$ is trivial on
$V(\mathrm{k})$, so we may consider $V(\mathrm{k})$ as a
$\Gamma_i$-module via the above twisted action.

\subsection{Teitelbaum type $L$-invariants} \label{ss:Teitelbaum}

Chida, Mok and Park \cite{CMP} defined Teitelbaum type $L$-invariant
for automorphic forms $\mathbf{f}\in
S^{\bar{B}}_{\mathrm{k}}(\Sigma)$ satisfying the condition (CMP)
given in the introduction:
\begin{equation*}
\mathbf{f} \text{ is new at } \fp \text{ and } U_\fp \mathbf{f} =
\CN \fp ^{w/2} \mathbf{f}.
\end{equation*}

\noindent We recall their construction below.

We attach to each $f_i$ a $\Gamma_i$-invariant
$V(\mathrm{k})$-valued cocycle $c_{f_i}$, where $\Gamma_i$ acts on
$V(\mathrm{k})$ via $\star$. For $e=(s,t)\in \mathrm{E}(\CT)$,
represent $s$ and $t$ by lattices $L_s$ and $L_t$ such that $L_s$
contains $L_t$ with index $\CN \mathfrak{p}$. Let
$g_e\in\GL(2,F_\fp)$ be such that $g_e(\mathfrak{o}_{F_\fp}^2)=L_s$.
Then we define $c_{f_i}(e):=|\det(g)|_\fp^{\frac{w-2}{2}} g_e\star
f_i(g_e)$. If $\mathbf{f}$ satisfies (CMP), then $c_{f_i}$ is in
$C^1_\har(V(\mathrm{k}))^{\Gamma_i}$ \cite[Proposition 2.7]{CMP}.
Thus we obtain a vector of harmonic cocycles
$c_{\mathbf{f}}=(c_{f_1},\cdots, c_{f_h})$.

For each $c\in C^1_\har(V(\mathrm{k}))^{\Gamma_i}$ we define
$\kappa^{\mathrm{sch}}_c$ to be the following function on $\Gamma_i$
with values in $V(\mathrm{k})$: fixing some $v\in \mathrm{V}(\CT)$,
for each $\gamma\in\Gamma_i$, we put
$$ \kappa^{\mathrm{sch}}_c(\gamma) := \sum_{e:v\rightarrow \gamma v} c(e)
$$ where $e$ runs over the edges in the geodesic joining $v$ and $\gamma
v$. As $c$ is $\Gamma_i$-invariant, $\kappa_c^{\mathrm{sch}}$ is a
$1$-cocycle. Furthermore the class of $\kappa^{\mathrm{sch}}_c$ in
$H^1(\Gamma_i, V(\mathrm{k}))$ is independent of the choice of $v$.
Hence we obtain a map
$$ \kappa^{\mathrm{sch}}: \bigoplus_{i=1}^h C^1_\har(
V(\mathrm{k}))^{\Gamma_i}\rightarrow \bigoplus_{i=1}^h H^1(\Gamma_i,
V(\mathrm{k})) . $$ By \cite[Proposition 2.9]{CMP}
$\kappa^{\mathrm{sch}}$ is an isomorphism.

For each $\sigma: F_\fp\rightarrow L_\fP$, let $L_{\fP,\sigma}(k,v)$
be the dual of $V_\sigma(k,v)$ with the right action of $\GL(2,
F_\fp)$: if $g\in \GL(2,F_\fp)$, $P' \in L_{\fP, \sigma}(k,v)$ and
$P\in V_\sigma(k,v)$, then $ \langle P' , g\cdot P \rangle = \langle
P'|_g , P\rangle $. We realize $L_{\fP,\sigma}(k,v)$ by the same
space as $V_\sigma(k,v)$, with the pairing $$\langle
X_\sigma^jY_\sigma^{k-2-j}, X_\sigma^{j'}Y_\sigma^{k-2-j'} \rangle =
\left\{ \begin{array}{ll} 1 & \text{ if }j=j' \\ 0 & \text{ if
}j\neq j' \end{array}\right.$$ and the right $\GL(2, F_\fp)$-action
$$P|_\wvec{a}{b}{c}{d} = P( \sigma(a) X_\sigma + \sigma(b) Y_\sigma, \sigma(c) X_\sigma + \sigma(d) X_\sigma  ).
$$

Put
\begin{equation} \label{eq:L-sigma}
L_\fP(\mathrm{k})^\tau
:=\bigotimes_{\sigma\neq\tau}L_{\fP,\sigma}(k_\sigma,(w-k_\sigma)/2)
\end{equation}
with the right action of $\GL(2,F_\fp)$, where the tensor product is
taken over $L_\fP$.

For each harmonic cocycle $c\in C^1_\har( V(\mathrm{k}))^{\Gamma_i}$
the method of Amice-Velu and Vishik allows one to define the
$V(\mathrm{k})^\tau$-valued rigid analytic distribution
$\mu_{c}^\tau$ on $\RP^1(F_\fp)$ such that the value of $ \int_{U_e}
t^j \mu_{c}^\tau(t) \in V(\mathrm{k})^\tau$ ($0\leq j\leq k_\tau-2$)
satisfies
$$ \langle Q, \int_{U_e} t^j \mu_{c}^\tau(t)\rangle= \frac{ \langle X_\tau^j Y_\tau^{k_\tau-2-j}\otimes Q, c(e)\rangle}{\binc{k_\tau-2}{j}}   $$ for each $Q\in
L_\fP(\mathrm{k})^\tau$. \footnote{There are minors in the
definitions of $\mu_c^\tau$ and $\lambda^\tau_c$ in \cite{CMP}. See
\cite{Teit} Definition 6 and the paragraph before Proposition 9.}

Using $\mu_{c}^\tau$ we obtain a $V(\mathrm{k})^\tau$-valued rigid
analytic function $g^\tau_c$, precisely a global section of
$V(\mathrm{k})^\tau\otimes_{\tau, F_\fp}
\CO_{\CH,\widehat{F_\fp^\ur}}$ by
$$ g^\tau_{c}(z) =\int_{\RP^1(F_\fp)}\frac{1}{z-t} \mu_{c}^\tau(t)
$$ for $z\in \CH_{\widehat{F_\fp^\ur}}$. The function $g_c^\tau$ satisfies the
transformation property: for $\gamma\in \tilde{\Gamma}_i$, let
$\wvec{a}{b}{c}{d}$ be the image of $\gamma$ in $\bar{B}_\fp\cong
\GL(2,F_\fp)$; then
$$ g_c^\tau(\gamma\cdot z) = |\Nrd_{B/F}\gamma|_\fp^{\frac{w-2}{2}} (\Nrd_{B/F}\gamma )^{\frac{w-2-k_\tau}{2}}
 (cz+d)^{k_\tau} \gamma \cdot g_c^\tau(z) .
$$

Consider $V(\mathrm{k})$ as an $F_\fp$-module via $\tau:
F_\fp\hookrightarrow L_\fP$. We define the $V(\mathrm{k})$-valued
cocycle $\lambda^\tau_c$ as follows.  Fix a point $z_0\in \CH$. For
each $\gamma\in\Gamma_i$ the value $\lambda_c^\tau(\gamma)$ is given
by the formula: for $Q\in L_\fP(\mathrm{k})^\tau$,
$$ \langle X_\tau^jY_\tau^{k_\tau-2-j}\otimes Q, \lambda_c^\tau(\gamma)\rangle =  \binc{k_\tau-2}{j}\langle Q, \int_{z_0}^{\gamma z_0} z^j g_c^\tau(z)\mathrm{d}z \rangle
$$ $(0\leq j\leq k_\tau-2)$, where the integral is the branch of Coleman's integral chosen in Section
\ref{sec:cover-hodge}. Then $\lambda_c^\tau$ is a $1$-cocycle on
$\Gamma_i$, and the class of $\lambda_c^\tau$ in $H^1(\Gamma_i,
V(\mathrm{k}))$, denoted by $[\lambda_c^\tau]$, is independent of
the choice of $z_0$. This defines a map $$
\kappa^{\mathrm{col},\tau}: \bigoplus_{i=1}^h C^1_\har(\CT,
V(\mathrm{k}))^{\Gamma_i}\rightarrow \bigoplus_{i=1}^h H^1(\Gamma_i,
V(\mathrm{k})), \hskip 10pt (c_i)_i \mapsto
([\lambda_{c_i}^\tau])_i.
$$

As $\kappa^{\mathrm{sch}}$ is an isomorphism, for each $\tau$ there
exists a unique $\ell_\tau\in L_\fP$ such that
$$\kappa^{\mathrm{col},\tau}(c_{\mathbf{f}})=\ell_\tau
\kappa^{\mathrm{sch}}(c_{\mathbf{f}}).$$ The {\it Teitelbaum type
$L$-invariant} of $\mathbf{f}$, denoted by $\CL_T(\mathbf{f})$, is
defined to be the vector $(\ell_\tau)_\tau$ \cite[Section 3.2]{CMP}.
We also write $\CL_{T,\tau}(\mathbf{f})$ for $\ell_\tau$.

\section{Comparing $L$-invariants} \label{sec:compare}

Let $B$, $\bar{B}$, $G$ and $\bar{G}$ be as before. Let $\fn^-$ be
the conductor of $\bar{B}$. By our assumption on $\bar{B}$,
$\fp\nmid \fn^-$ and the conductor of $B$ is $\fp\fn^-$. Let $\fn^+$
be an ideal of $\mathfrak{o}_F$ that is prime to $\fp\fn^-$, and put
$\fn:=\fp\fn^+\fn^-$.

For any prime ideal $\fl$ of $\mathfrak{o}_F$, put
$$ \bar{R}_\fl := \left\{  \begin{array}{ll}\text{an maximal compact open subgroup of } \bar{B}_\fl^\times & \text{ if } \fl \text{ is prime to }
\fn, \\
 \text{the maximal compact open subgroup of }
\bar{B}_\fl^\times & \text{ if } \fl \text{ divides } \fn^-,
\\
1 + \text{ an Eichler order of } \bar{B}_\fl \text{ of level }
\fl^{\val_\fl(\fp\fn^+)} & \text{ if } \fl \text{ divides }\fp\fn^+.
 \end{array} \right. $$ Let $\bar{\Sigma}=\Sigma(\fp\fn^+,\fn^-)$ be the
level $\prod_\fl \bar{R}_\fl $. We write
$S^{\bar{B}}_{\mathrm{k}}(\fp\fn^+,\fn^-)$ for
$S^{\bar{B}}_{\mathrm{k}}(\Sigma(\fp\fn^+,\fn^-))$. Similarly we
define $\Sigma=\Sigma(\fn^+,\fp\fn^-)$, a compact open subgroup of
$G(\BA_f)$. Let $S_\mathrm{k}^B(\fn^+, \fp\fn^-)$ be the space of
modular forms on the Shimura curve $M$ of weight $\mathrm{k}$ and
level $\Sigma$.

Let $\mathrm{k}=(k_1,\cdots, k_g,w)$ be a multiweight such that
$k_1\equiv \cdots k_g\equiv w \mod 2$ and $k_1,\cdots ,k_g $ are all
even and larger than $2$. Let $f_\infty$ be a (Hilbert)  eigen cusp
newform of weight $\mathrm{k}$ and level $\fn$ that is new at
$\mathfrak{p}\mathfrak{n}^-$. Let $\mathbf{f}\in
S_{\mathrm{k}}^{\bar{B}}(\fp\fn^+, \fn^-)$ (resp. $f_B\in
S_{\mathrm{k}}^{B}(\fn^+, \fp\fn^-)$) be an eigen newform
corresponding to $f_\infty$ by the Jacquet-Langlands correspondence;
$\mathbf{f}$ (resp. $f_B$) is unique up to scalars.

We further assume that $\mathbf{f}$ satisfies (CMP), so that we can
attach to $\mathbf{f}$ the Teitelbaum type $L$-invariant
$\CL_T(\mathbf{f})$. We define $\CL_T(f_\infty)$ to be
$\CL_T(\mathbf{f})$. The goal of this section is to compare
$\CL_{FM}(f_\infty)$ and $\CL_T(f_\infty)$.

Let $L$ be a (sufficiently large) finite extension of $F$ that
splits $B$ and contains all Hecke eigenvalues acting on $f_\infty$.
Let $\lambda$ be an arbitrary place of $L$.

\begin{lem} \label{lem:saito}
\cite[Lemma 3.1]{Saito} There is an isomorphism
$$ H^1_{\mathrm{et}}(M_{\overline{F}}, \CF(\mathrm{k})_\lambda) \simeq
\bigoplus_{f'} \pi^\infty_{f', L(f')}\otimes_{L(f')}
(\bigoplus_{\lambda'|\lambda} \rho_{f', \lambda'})
$$ of representations of $G(\BA_f)\times \Gal(\overline{F}/F)$ over
$L_\lambda$. Here $f'$ runs through the conjugacy classes over $L$,
up to scalars, of eigen newforms of multiweight $\mathrm{k}$ that
are new at primes dividing $\fp\fn^-$. The extension of $L$
generated by the Hecke eigenvalues acting on $f'$ is denoted by
$L(f')$, and $\lambda'$ runs through places of $L(f')$ above
$\lambda$.
\end{lem}

By the strong multiplicity one theorem (cf. \cite{PS}) there exists
a primitive idempotent $e_{\mathbf{f}}\in \BT_{\bar{\Sigma}}$ such
that $e_{\mathbf{f}}\BT_{\bar{\Sigma}}=Le_{\mathbf{f}}$ and
$e_{\mathbf{f}}\cdot
S_{\mathrm{k}}^{\bar{B}}(\Sigma(\fp\fn^+,\fn^-)) = L\cdot
\mathbf{f}$. Lemma \ref{lem:saito} tells us that
$e_{\mathbf{f}}\cdot H^1_{\mathrm{et}}(M_{\overline{F}},
\CF(\mathrm{k})_\lambda)^{\Sigma}$ is exactly $\rho_{f_B,\lambda}$,
the $\lambda$-adic representation of $\Gal(\overline{F}/F)$ attached
to $f_B$. By Carayol's construction of $\rho_{f_\infty,\lambda}$
\cite{Car2}  $\rho_{f_\infty,\lambda}$ coincides with
$\rho_{f_B,\lambda}$.

Now we take $\lambda$ to be a place above $\fp$, denoted by $\fP$.

Recall that in Section \ref{ss:apply} and Section
\ref{sec:Teit-L-inv} we associate to $\bar{\Sigma}$ the groups
$\widetilde{\Gamma}_{i,0}, \widetilde{\Gamma}_i,\Gamma_i,
\Gamma_{i,0}$ ($i=1,\cdots,h$). By (\ref{eq:decomp}) $X_{\Sigma}$ is
isomorphic to $\coprod_{i} X_{\Gamma_{i,0}}$, where
$X_{\Gamma_{i,0}}=\Gamma_{i,0}\backslash \CH_{\widehat{F_\fp^\ur}}$.

\begin{thm} \label{thm:semistable}
Let $f_B$ be as above. Then $\rho_{f_B, \fP, \fp}$ is a semistable
$($non-crystalline$)$ representation of
$\Gal(\overline{F}_\fp/F_\fp)$, and the filtered
$(\varphi_q,N)$-module $D_{\st, F_\fp}(\rho_{f_B,\fP,\fp})$ is a
monodromy $L_\fP$-module.
\end{thm}
\begin{proof}
To show that $\rho_{f_B, \fP, \fp}$ is semistable, we only need to
prove that $H^1_{\mathrm{et}}((X_{\Sigma})_{{\bar{F}}_\fp},
\CF(\mathrm{k}))$ is semistable, since $\rho_{f_B, \fP, \fp}$ is a
subrepresentation of
$H^1_{\mathrm{et}}((X_{\Sigma})_{{\bar{F}}_\fp}, \CF(\mathrm{k}))$.
But this follows from Proposition \ref{prop:semistable} and the fact
that $X_{\Sigma}$ is semistable.

Next we prove that $D_{\st, F_\fp}(\rho_{f_B,\fP,\fp})$ is a
monodromy $L_\fP$-module. We only need to consider $D_{\st,
 {\widehat{F_\fp^\ur}}}(\rho_{f_B,\fP,\fp})$ instead.

Twisting $f_B$ by a central character we may assume that $w=2$.

Being a Shimura variety, $N_{E}$ is a family of varieties. But in
the following we will use $N_E$ to denote any one in this family
that corresponds to a level subgroup whose $\fp$-factor is
$\scrO_{E_\fp}^\times$. By the proof of Lemma \ref{lem:fil-mod-1},
any geometric point of $(N_E)_{ \widehat{F_\fp^\ur}}$ is defined
over $\widehat{F_\fp^\ur}$. In other words, $(N_E)_{
\widehat{F_\fp^\ur}}$ is the product of several copies of $\Spec(
\widehat{F_\fp^\ur})$.

Let $\pr_1$ be the projection $X_{\Sigma}\times (N_E)_{\widehat
{F_\fp^\ur}}\rightarrow X_{\Sigma}$. By Corollary
\ref{cor:fil-isocry} the filtered $\varphi_q$-isocrystal attached to
$\pr_1^*\mathcal{F}(\mathrm{k})$ is
$\pr_1^*\mathscr{F}(\mathrm{k})$. Note that
$$ H^1_{\mathrm{et}}((X_{\Sigma}\times
(N_E)_{ \widehat{F_\fp^\ur}})_{\widehat{\overline{F}}_\fp},
\pr_1^*\mathcal{F}(\mathrm{k}))=
H^0_{\mathrm{et}}((N_{E})_{\widehat{\overline{F}}_\fp},\BQ_p)\otimes_{\BQ_p}
H^1_{\mathrm{et}}((X_{\Sigma})_{\widehat{\overline{F}}_\fp},
\mathcal{F}(\mathrm{k})) . $$ The
$\Gal(\overline{F}_\fp/F_\fp^\ur)$-representation
$H^0_{\mathrm{et}}((N_{E})_{\widehat{\overline{F}}_\fp},\BQ_p)$ is
crystalline and the associated filtered $\varphi_q$-module is
$H^0_\dR((N_E)_{ \widehat{F^\ur_\fp}},\BQ_p)$ with trivial
filtration. Let $H^0$ denote this filtered $\varphi_q$-module for
simplicity. As a consequence, we have an isomorphism of filtered
$(\varphi_q,N)$-modules
\begin{equation} H^0\otimes _{\BQ_p} D_{\st,
\widehat{F_\fp^\ur}}(H^1_{\mathrm{et}}(
(X_{\Sigma})_{\widehat{\overline{F}}_\fp}, \mathcal{F}(\mathrm{k})))
= H^0 \otimes _{\BQ_p} H^1_\dR( X_{\Sigma} , \mathscr{F}(\mathrm{k})
) .
\end{equation}

Using the decomposition (\ref{eq:decom-isocrystal}), for each
embedding $\tau:F_\fp\hookrightarrow L_\fP$ we put
$$ H^1_{\dR,\tau} (
X_{\Sigma} , \mathscr{F}(\mathrm{k}) ) := \mathbb{H}^1(X_{\Sigma} ,
\CV(\mathrm{k})\otimes_{\tau, F_\fp } \Omega^\bullet_{X_{\Sigma}} ).
$$ In Section \ref{ss:Teitelbaum} we attached to
$\mathbf{f}=(f_1,\cdots, f_h)$ an $h$-tuple
$g^\tau=(g_1^\tau,\cdots, g_h^\tau)$.  Let $M_{\tau}(\mathbf{f})$
denote the $L_\fP$-subspace of $\bigoplus_i
H^1_{\dR,\tau}(X_{\Gamma_{i,0}}, \mathscr{F}(\mathrm{k}))$ generated
by the element
$$\omega_{\mathbf{f}}^\tau = \Big( g^\tau_i(z)(z X_\tau +
Y_\tau)^{k_\tau-2} \rd z  \Big)_{1\leq i\leq h}.$$ 
(Note that, when $w=2$, the twisted action $\star$ in Section
\ref{sec:Teit-L-inv} coincides with the original action.) Therefore,
we have
$$ e_\mathbf{f} \cdot \Fil^{\frac{w+k_\tau}{2}-2}H^1_{\dR,\tau}(X_{\Sigma}, \mathscr{F}(\mathrm{k}))
\supseteq M_\tau(\mathbf{f}). $$

Consider the pairing $<\cdot,\cdot>$ on $V(\mathrm{k})$ defined by
$$<Q_1,Q_2> \ = \text{the coefficient of }\prod_\sigma (X_\sigma
Y_\sigma)^{k_\sigma-2} \text{ in } Q_1Q_2.$$ It is perfect  and
induces a perfect pairing on $H^1_{\dR,\tau}(X_{\Sigma},
\mathscr{F}(\mathrm{k}))$. With respect to this pairing
$\Fil^{\frac{w-k_\tau}{2}+1} H^1_{\dR,\tau}(X_{\Sigma},
\mathscr{F}(\mathrm{k}))$ is orthogonal to
$\Fil^{\frac{w+k_\tau}{2}-2} H^1_{\dR,\tau}(X_{\Sigma},
\mathscr{F}(\mathrm{k}))$. As $e_\mathbf{f} \cdot
H^1_{\dR,\tau}(X_{\Sigma}, \mathscr{F}(\mathrm{k}) )$ is of rank $2$
over $L_\fP$, we obtain
$$ e_\mathbf{f} \cdot \Fil^{\frac{w-k_\tau}{2}+1}H^1_{\dR,\tau}(X_{\Sigma}, \mathscr{F}(\mathrm{k}))=
e_\mathbf{f} \cdot
\Fil^{\frac{w+k_\tau}{2}-2}H^1_{\dR,\tau}(X_{\Sigma},
\mathscr{F}(\mathrm{k}) ) = M_\tau(\mathbf{f}). $$ Thus
$$H^0\otimes_{\BQ_p}\pr_1^* M_\tau(\mathbf{f}) = H^0\otimes_{\BQ_p}\Fil^{\frac{w+\min_\tau\{ k_\tau\}}{2}-2} D_{\st,
\widehat{F_\fp^\ur}}(\rho_{f_B,\fP,\fp})_\tau.$$

Let $\iota^\tau$ and $I^\tau$ be the operators attached to the sheaf
$\mathscr{F}(\mathrm{k})$ over $X_{\Sigma}$ (see Section
\ref{ss:apply}). Here, the superscript $\tau$ is used to emphasize
the embedding $\tau:F_\fp\hookrightarrow L_\fP$. Proposition
\ref{prop:monodromy} tells us that the monodromy $N$ on
$H^1_{\dR,\tau}(X_{\Sigma}, \mathscr{F}(\mathrm{k}))$
coincides with $\iota^\tau\circ I^\tau$. By Proposition \ref{prop:deShalit} 
the kernel of $N$ is $$\iota^\tau\circ \delta^{-1} (\bigoplus_i
H^1(\Gamma_{i,0}, V(\mathrm{k}))).$$ So, by (\ref{eq:hodge-cover})
the restriction of $N$ to $M_{\tau}(\mathbf{f})$ is injective.
Hence,
$$\ker(N)\cap \Fil^{\frac{w+\min_\tau\{ k_\tau\}}{2}-2}D_{\st, F_\fp^\ur}(\rho_{f_B,\fP,\fp})_\tau=0,$$ as desired.
\end{proof}

Let $P^\tau$ be the operator attached to $\mathscr{F}(\mathrm{k})$
(see Section \ref{ss:apply}).

\begin{lem}\label{lem:PI-comp-kappa} Let $\omega_{\mathbf{f}}^\tau$ be as in
the proof of Theorem \ref{thm:semistable}. Then $$
P^\tau(\omega_{\mathbf{f}}^\tau)=\kappa^{\mathrm{col},\tau}(c_{\mathbf{f}})
, \hskip 20pt I^\tau(\omega_{\mathbf{f}}^\tau)=c_{\mathbf{f}}. $$
\end{lem}
\begin{proof} The first formula comes from the definitions.

The proof of the second formula is similar to that of \cite[Theorem
3]{Teit}. 
Let $\mu^\tau_i$ ($i=1,\cdots,h$) be the rigid analytic
distributions on $\mathrm{P}^1(F_\fp)$ coming from $c_\mathbf{f}$
(see Section \ref{sec:Teit-L-inv}). Recall that
$$ g_{i}^\tau(z) = \int_{\RP^1(F_\fp)} \frac{1}{z-t}\mu_i^\tau (t) .
$$ For each edge
$e$ of $\CT$ let $B(e)$ be the affinoid open disc in $\RP^1(\BC_p)$
that corresponds to $e$. Assume that $B(e)$ meets the limits set
$\RP^1(F_\fp)$ in a compact open subset $U(e)$. Put
$$ g^\tau_{i,e}(z)=\int_{U(e)} \frac{1}{z-t} \mu_i^\tau(t).$$ Let
$a(e)$ be a point in $U(e)$. Expanding $\frac{1}{z-t}$ at $a(e)$ we
obtain that
$$ g^\tau_{i,e}(z) =\sum_{n=0}^{+\infty} \frac{1}{(z-a(e))^{n+1}} \int_{U(e)}  (t-a(e))^n
\mu^\tau_i(t) ,
$$ and thus $g^\tau_{i,e}(z)$ converges on the complement of
$B(e)$. Note that $g^\tau_i-g^\tau_{i,e}$ is analytic on $B(e)$. So,
we have
\begin{eqnarray*} && I^\tau(g_i^\tau (z X_\tau +
Y_\tau)^{k_\tau-2}\rd z) (e) = \Res_e \Big(g_i^\tau (z X_\tau +
Y_\tau)^{k_\tau-2}\rd z \Big) = \Res_e \Big( g^\tau_{i,e}
(zX_\tau+Y_\tau)^{k_\tau-2}\rd z \Big) \\ && = \Res_e \Big(
\int_{U(e)} \frac{ ( z X_\tau + Y_\tau)^{k_\tau-2} } {z-t}
\mu_i^\tau(t) \Big) = \int_{U(e)} ( t X_\tau + Y_\tau)^{k_\tau-2}
\mu_i^\tau(t) = c_{f_i}(e),
\end{eqnarray*} where the fourth equality follows from the fact that $\Res_e$ commutes with $\int_{U(e)} \ \cdot \ \mu_i^\tau(t)$.
\end{proof}

\begin{thm} \label{thm:main-text} Let $f_\infty$ be as above. Then
$\CL_{FM}(f_\infty)=\CL_T(f_\infty)$.
\end{thm}
\begin{proof} Twisting $f_\infty$ by a central character we may
assume that $w=2$.

Put $D_\tau=H^0\otimes_{\BQ_p} e_{\mathbf{f}}
H^1_{\dR,\tau}(X_{\Sigma}, \mathscr{F}(\mathrm{k}) )$. Note that
$N=\iota^\tau\circ I^\tau$. As the kernel of $N$ coincides with the
image of $\iota^\tau\circ \delta^{-1}$ and $P^\tau$ splits
$\iota^\tau\circ \delta^{-1}$, we have $D_\tau = \ker(N) \oplus
\ker(P^\tau)$. Write $\omega_{\mathbf{f}}^\tau= x+y$ according to
this decomposition.  Then \begin{equation}\label{eq:proof-a}
\iota^\tau\circ \delta^{-1}\circ P^\tau (\omega^\tau_\mathbf{f}) = x
.\end{equation} By the proof of Theorem \ref{thm:semistable}, $y$ is
non-zero and so $N(y)\neq 0$.

By Lemma \ref{lem:PI-comp-kappa} and the definition of Teitelbaum
type $L$-invariant, $\CL_{T,\tau}(f_\infty)$ is characterized by the
property \begin{equation} \label{eq:proof-b}
(P^\tau-\CL_{T,\tau}(f_\infty) \epsilon\circ
I^\tau)\omega^\tau_{\mathbf{f}} = 0, \end{equation} where $\epsilon$
is the map defined by (\ref{eq:sch}) which coincides with
$\kappa^{\mathrm{sch}}$.
 As
$\delta^{-1}\circ\epsilon =-\mathrm{id}$ and $\iota^\tau\circ
I^\tau=N$, we have \begin{equation}\label{eq:proof-c} \iota^\tau
\circ \delta^{-1} \circ \epsilon \circ I^\tau
(\omega^\tau_\mathbf{f}) = - N(\omega_\mathbf{f}^\tau) .
\end{equation} By (\ref{eq:proof-a}), (\ref{eq:proof-b}) and (\ref{eq:proof-c}) we get
\begin{equation}\label{eq:L-Teit}
\CL_{T,\tau}(f_\infty)N(\omega_\mathbf{f}^\tau) + x = 0.
\end{equation}

By the definition of Fontaine-Mazur $L$-invariant,
$\CL_{FM,\tau}(f_\infty)$ is characterized by the property
\begin{equation} \label{eq:L-FM}
y-\CL_{FM,\tau}(f_\infty)N(y) \in
H^0\otimes_{\BQ_p}\Fil^{\frac{w+\min_\tau\{ k_\tau\}}{2}-2}
H^1_{\dR,\tau}(X_{\Sigma},  \mathscr{F}(\mathrm{k})).
\end{equation}

 Combining  (\ref{eq:L-Teit}) and (\ref{eq:L-FM})  we obtain
{\allowdisplaybreaks \begin{eqnarray*} &&
(\CL_{FM,\tau}(f_\infty)-\CL_{T,\tau}(f_\infty))N(y) \\ & = &
\CL_{FM,\tau}(f_\infty)N(y) -
\CL_{T,\tau}(f_\infty)N(\omega_\mathbf{f}^\tau)
\\ & \in & \omega_{\mathbf{f}}^\tau + H^0\otimes_{\BQ_p}\Fil^{\frac{w+\min_\tau\{ k_\tau\}}{2}-2}
H^1_{\dR,\tau}(X_{\Sigma},  \mathscr{F}(\mathrm{k}) ) \\ &=&
H^0\otimes_{\BQ_p}\Fil^{\frac{w+\min_\tau\{ k_\tau\}}{2}-2}
H^1_{\dR,\tau}(X_{\Sigma},  \mathscr{F}(\mathrm{k}) ).
\end{eqnarray*} }

\noindent But $N(y)$ is in $\ker(N)$ and is non-zero, and by Theorem
\ref{thm:semistable}
$$\ker(N)\cap H^0\otimes_{\BQ_p}\Fil^{\frac{w+\min_\tau\{
k_\tau\}}{2}-2} H^1_{\dR,\tau}(X_{\Sigma}, \mathscr{F}(\mathrm{k}) )
= 0 .$$ Therefore
$$\CL_{FM,\tau}(f_\infty)-\CL_{T,\tau}(f_\infty)=0,$$ as wanted.
\end{proof}


\end{document}